\documentclass[12 pt]{article}
\usepackage{latexsym, amssymb, amsmath, amsthm, geometry, changebar, epsfig}

\begin{document}

\newcommand{\eps}{\varepsilon}

\newcommand{\bei}{\begin{itemize}}
\newcommand{\eei}{\end{itemize}}
\newcommand{\beq}{\begin{equation}}
\newcommand{\eeq}{\end{equation}}
\newcommand{\beqr}{\begin{eqnarray}}
\newcommand{\eeqr}{\end{eqnarray}}
\newcommand{\beqrn}{\begin{eqnarray*}}
\newcommand{\eeqrn}{\end{eqnarray*}}
\newcommand{\brr}{\begin{array}}
\newcommand{\err}{\end{array}}
\newcommand{\bef}{\begin{figure}}
\newcommand{\eef}{\end{figure}}
\newcommand{\al}{\alpha}
\newcommand{\ga}{\gamma}

\newcommand{\del}{\delta}
\newcommand{\sig}{\sigma}

\newcommand{\Tau}{\mbox{\rm T}}

\newcommand{\disju}{\dot{\bigcup}}
\newcommand{\eqdf}{\stackrel{\mathrm{def}}{=}}
\newcommand{\Diag}{\mbox{Diag}}

\newcommand{\barA}{\overline{\mathcal{A}}}
\newcommand{\A}{\mathcal{A}}
\newcommand{\K}{\mathcal{K}}
\newcommand{\Si}{\mathcal{S}}
\newcommand{\Z}{\mathcal{Z}}
\newcommand{\B}{\mathring{\mathcal{B}}}
\newcommand{\E}{\mathcal{E}}
\newcommand{\G}{\mathcal{G}}
\newcommand{\zlj}{Z_{\ell, j}}
\newcommand{\Lp}{\mathcal{L}}
\newcommand{\ppi}{\mathcal{P}_\pi}
\newcommand{\hDelta}{\mathring{\Delta}}
\newcommand{\dlj}{\Delta_{\ell,j}}
\newcommand{\hdlj}{\hDelta_{\ell,j}}
\newcommand{\hlj}{H_{\ell,j}}
\newcommand{\wslj}{\W^s_{\ell,j}}
\newcommand{\wulj}{\W^u_{\ell,j}}
\newcommand{\F}{\mathring{F}}
\newcommand{\tF}{\tilde{F}}
\newcommand{\FN}{\overline{F}_N}
\newcommand{\T}{\mathring{T}}
\newcommand{\hf}{\mathring{f}}
\newcommand{\I}{\hat{I}}
\newcommand{\X}{\hat{X}}
\newcommand{\V}{\mathcal{V}}
\newcommand{\M}{\mathcal{M}}
\newcommand{\hM}{\mathring{M}}
\newcommand{\W}{\mathcal{W}}
\newcommand{\Ho}{\mathcal{H}}
\newcommand{\Q}{\mathcal{Q}}
\newcommand{\J}{\mathcal{J}}
\newcommand{\N}{\mathbb{N}}
\newcommand{\Na}{\mathcal{N}}
\newcommand{\R}{\mathbb{R}}
\newcommand{\ff}{\mathcal{F}}
\newcommand{\tQ}{\tilde{\Q}}
\newcommand{\ab}{\bar{a}}
\newcommand{\ve}{\varepsilon}
\newcommand{\vf}{\varphi}
\newcommand{\bvf}{\overline{\varphi}}
\newcommand{\bpsi}{\overline{\psi}}
\newcommand{\barF}{\overline{F}}
\newcommand{\bDelta}{\overline{\Delta}}
\newcommand{\bLp}{\overline{\Lp}}
\newcommand{\bpi}{\overline{\pi}}
\newcommand{\bm}{\overline{m}}
\newcommand{\bmu}{\overline{\mu}}
\newcommand{\bnu}{\overline{\nu}}
\newcommand{\bareta}{\overline{\eta}}
\newcommand{\brho}{\overline{\rho}}
\newcommand{\bpa}{\overline{\pa}}
\newcommand{\bp}{\overline{P}}
\newcommand{\bdelta}{\overline{\delta}}
\newcommand{\bH}{\overline{H}}
\newcommand{\tm}{\tilde{m}}
\newcommand{\hm}{\hat{m}}
\newcommand{\tG}{\tilde{\G}}
\newcommand{\tf}{\tilde{f}}
\newcommand{\tphi}{\tilde{\varphi}}
\newcommand{\tpsi}{\tilde{\psi}}
\newcommand{\tP}{\tilde{P}}
\newcommand{\tH}{\tilde{H}}
\newcommand{\td}{\tilde{d}}
\newcommand{\tx}{\tilde{x}}
\newcommand{\ty}{\tilde{y}}
\newcommand{\tnu}{\tilde{\nu}}
\newcommand{\hnu}{\hat{\nu}}
\newcommand{\hmu}{\hat{\mu}}
\newcommand{\tmu}{\tilde{\mu}}
\newcommand{\tT}{\tilde{T}}
\newcommand{\heta}{\hat{\eta}}
\newcommand{\teta}{\tilde{\eta}}
\newcommand{\accim}{{\em a.c.c.i.m.}}
\newcommand{\diam}{\mbox{diam}}
\newcommand{\pa}{\mathcal{P}}
\newcommand{\tpa}{\tilde{\mathcal{P}}}
\newcommand{\lip}{\mbox{\tiny Lip}}
\newcommand{\loc}{\mbox{\scriptsize loc}}

\newcommand{\flag}[1]{\textbf{******[#1]******}}
\newcommand{\ra}{\vartheta}

\newtheorem{theorem}{Theorem}[section]
\newtheorem{thm}{Theorem}
\newtheorem{lemma}[theorem]{Lemma}
\newtheorem{sublemma}[theorem]{Sublemma}
\newtheorem{convention}[theorem]{Convention}
\newtheorem{proposition}[theorem]{Proposition}
\newtheorem{remark}[theorem]{Remark}
\newtheorem{definition}[theorem]{Definition}
\newtheorem{corollary}[theorem]{Corollary}
\newtheorem{assumption}{Assumption}
\newtheorem{claim}{Claim}[section]
\setlength{\unitlength}{1cm}

\title{Escape Rates and Physically Relevant Measures for Billiards with
Small Holes}
\author{Mark Demers\thanks{Department of Mathematics and Computer Science,
Fairfield University.  Email:  mdemers@mail.fairfield.edu.  This
research is partially supported by NSF grant DMS-0801139.} \and Paul
Wright\thanks{Department of Mathematics, University of Maryland.
Email:  paulrite@math.umd.edu.  This research is partially supported
by an NSF Mathematical Sciences Postdoctoral Research Fellowship.  
This author would also like to thank The Courant Institute of
Mathematical Sciences, New York University, where he was affiliated when
this project began.}
\and Lai-Sang Young\thanks{Courant Institute of Mathematical
Sciences, New York University. Email: lsy@cims.nyu.edu. This
research is partially supported by a grant from the NSF.
\newline  The authors would like to thank MSRI, Berkeley, and ESI, Vienna, where
part of this work was carried out.
}}

\maketitle

\begin{abstract}
We study the billiard map corresponding to a periodic Lorentz gas in 
2-dimensions in the presence of small holes in the table.  We allow 
holes in the form of open sets away
from the scatterers as well as segments on the
boundaries of the scatterers.   For a large class of smooth initial distributions, 
we establish the existence of a common escape rate and normalized
limiting distribution.  This limiting
distribution is conditionally invariant and is the natural analogue
of the SRB measure of a closed system.
Finally, we prove that as the size of the hole tends to
zero, the limiting distribution converges to the smooth invariant measure of
the billiard map.
\end{abstract}

%\large

%\tableofcontents

This paper is about {\it leaky dynamical systems}, or {\it dynamical systems with
holes}. Consider a dynamical system defined by a map or a flow on a phase space
$M$, and let $H \subset M$ be a hole through which orbits escape, that is to say,
once an orbit enters $H$, we stop considering it from that point on.
Starting from an initial probability distribution $\mu_0$ on $M$, mass will leak out of
the system as it evolves. Let $\mu_n$ denote the distribution remaining at time $n$.
The most basic question one can ask about a leaky system is its
{\it rate of escape}, i.e. whether $\mu_n(M) \sim \ra^n$ for some $\ra$.
Another important question concerns the nature of the remaining distribution.
One way to formulate that is to normalize $\mu_n$, and to inquire about properties
of $\mu_n/\mu_n(M)$ as $n$ tends to infinity. Such limiting distributions, when
they exist, are not invariant; they are {\it conditionally invariant}, meaning
they are invariant up to a normalization.
Comparisons of systems with
small holes with the corresponding {\it closed} systems, i.e. systems for which
the holes have been plugged, are also natural. These are some of the questions
we will address in this paper.

We do not consider these questions in the abstract, however; for a review paper
in this direction, see \cite{demers young}. Our context here is that of billiard systems
with small holes. Specifically,
we carry out our analysis for the collision map of
a 2-dimensional periodic Lorentz gas, and expect our results to be
extendable to other dispersing billiards. Our holes are ``physical" holes,
in the sense that they are derived from holes in the physical domain of
the system, i.e., the billiard table: we consider
both convex holes away from the scatterers and holes that live on the boundaries
of the scatterers. The holes considered in this paper are very small,
but their placements are immaterial.
For these leaky systems, we prove that there is a common rate
of escape and a common limiting distribution for a large class of natural
initial distributions including those with densities with respect to Liouville measure.
These conditionally invariant measures, therefore, can be viewed as characteristic
of the leaky systems in question, in a way that is analogous to physical measures or
SRB measures for closed systems. We show, in fact, that as hole size tends to
zero, these measures tend to the natural invariant measure of the corresponding
closed billiard system.

Our proof involves constructing a Markov tower extension with a special property
over the billiard map, the new requirement being that it {\it respects the hole}.
Let us backtrack a little for readers not already familiar with these ideas:
In much the same way that Markov partitions have proved to be very useful
in the study of Anosov and Axiom A diffeomorphisms, it was shown, beginning
with [Y] and continued in a number of other papers, that many systems with sufficiently
strong hyperbolic properties (but which are not necessarily uniformly hyperbolic)
admit countable Markov extensions. Roughly speaking, these extensions behave
like countable state Markov chains ``with nonlinearity"; they have considerably
simpler structures than the original dynamical system.  The idea behind this work
is that escape dynamics are much simpler in a Markov setting when the hole corresponds to a collection of ``states"; this is what we mean by the Markov
extension ``respecting the hole."
All this is not for free, however. We pay a price with a somewhat elaborate construction
of the tower, and again when we pass the information back to the billiard system,
in exchange for having a Markov structure to work with in the treatment of the hole.

There are advantages to this route of proof: First, once a Markov extension 
is constructed for a system, it can be used many times over for entirely 
different purposes. For the billiard 
maps studied here, these extensions were constructed in [Y]; 
our main task is to adapt them to holes. Second, once results on escape dynamics
are established on towers, they apply to all Markov extensions. 
Here, the desired results are already known in a special 
case, namely expanding towers  \cite{demers bruin}; we need to extend them
to the general, hyperbolic setting. {\it What we propose here is a unified,
generic approach for dealing with holes in dynamical systems}, 
one that can, in principle, be carried out for all systems that admit Markov towers. 
Such systems include logistic maps, rank one attractors including 
the H\'enon family, piecewise hyperbolic maps and other dispersing 
billiards in 2 or more dimensions.

Conditionally invariant measures were first introduced in 
probabilistic settings, namely
countable state Markov chains and topological Markov chains, 
beginning with
\cite{vere} and more recently in \cite{ferrari} and \cite{collet 3}.
In this setting, such measures are called quasi-stationary distributions
and the existence of a Yaglom limit corresponds to 
the limit $\mu_n/\mu_n(M)$, which we use here to identify a physical 
conditionally invariant measure for the leaky system.

The first works to study deterministic systems with holes took advantage of
finite Markov partitions.  These include: 
Expanding maps on $\R^n$ with holes which are elements of a finite Markov
partition \cite{pianigiani yorke, collet 1, collet 2};
Smale horseshoes \cite{cencova 1, cencova 2};
Anosov diffeomorphisms
\cite{chernov m 1, chernov m 2, chernov mt 1, chernov mt 2};
billiards with convex scatterers satisfying a non-eclipsing condition
\cite{lopes mark, richardson} and large parameter logistic maps whose critical
point maps out of the interval \cite{homburg young}. In the latter two, the
holes are chosen in such a way that the surviving dynamics are uniformly
expanding or hyperbolic with Markov partitions.  
First results
which drop Markov requirements on the map
include piecewise expanding maps of the interval
\cite{baladi keller, chernov bedem, liverani maume, demers exp, demers bruin};
Misiurewicz \cite{demers logistic} and Collet-Eckmann \cite{demers bruin} maps
with generic holes; and
piecewise uniformly hyperbolic maps \cite{demers liverani}.
The tower construction is used in the one-dimensional studies
\cite{demers exp, demers logistic,
demers bruin}. Typically a restriction on the size of the hole is introduced
in order to control the dynamics when a finite Markov partition is absent.

General conditions ensuring the existence of conditionally invariant
measures are first given in \cite{maume}.  The physical
relevance of such measures, however, is unclear without further 
qualifications.
As noted in \cite{demers young}, under very weak assumptions
on the dynamical system, many such
measures exist: for any prescribed rate of escape, one can construct
infinitely many conditionally invariant densities.
This is the reason for the
emphasis placed in this paper on the limit $\mu_n/\mu_n(M)$,
which identifies a unique, physically relevant conditionally
invariant measure.

\medskip
This paper is organized as follows: Our results are formulated in Sect.~1.
In Sects.~2 and 3, the geometry of billiard maps and holes are looked at
carefully as we modify previous constructions to give a generalized
horseshoe that respects the hole. Out of this horseshoe, a Markov tower
extension is constructed and results on escape dynamics on it proved; this
is carried out in Sects.~4 and 5. These results are passed back to the billiard system
in Sect.~6, where the remaining theorems are also proved.

%%%%%%%%%%%%%%%%%%%%%%%%%%%%%%%%%%
\section{Formulation of Results}

\subsection{Basic definitions}
\label{open def}

We consider a {\it closed} dynamical system defined by a self-map
$f$ of a manifold $M$, and let $H \subset M$ be a {\it hole} through
which orbits escape, {\it i.e.},  we stop considering an orbit once
it enters $H$. In this paper we are primarily concerned with holes
that are open subsets of the phase space; they are not too large and
generally not $f$-invariant. We will refer to the triplet $(f,M,H)$
as a {\it leaky system}.

First we introduce some notation. Let $\mathring M = M \backslash
H$. At least to begin with, let us make a formal distinction between
$f$ and $\mathring f = f |(\mathring M \cap f^{-1}\mathring M):
\mathring M \cap f^{-1}\mathring M \to \mathring M$, and write
$\mathring f^n = f^n | (\bigcap_{i=0}^n f^{-i}\mathring M)$. Let
$\eta$ be a probability measure on $\mathring M$. We define
$\mathring f_*\eta$ to be the measure on $\mathring M$ defined by
$(\mathring f_*\eta)(A) = \eta(\mathring f^{-1}A)$ for each Borel
set $A \subset \mathring M$. If $\eta$ is an initial distribution on
$\mathring M$, then $\eta^{(n)} := \mathring f^n_*\eta/|\mathring
f^n_*\eta|$ is the normalized distribution of points remaining in
$\mathring M$ after $n$ units of time.

Given an initial distribution $\eta$, the most basic question is the
rate at which mass is leaked out of the system. We define the {\it
escape rate} starting from $\eta$ to be $-\log \ra (\eta)$ where
$$
\log \ra (\eta) = \lim_{n\to \infty} \frac{1}{n} \log
\eta\left(\bigcap_{i=0}^n f^{-i}\mathring M \right) \qquad {\rm assuming \
such \ a \ limit \ exists.}
$$
Another basic object is the {\it limiting distribution}
$\eta^{(\infty)}$ defined to be $\eta^{(\infty)} =  \lim_{n \to
\infty} \eta^{(n)}$ if this weak limit exists. Of particular
interest is when there is a number $\ra_*$ and a probability measure
$\mu_*$ with the property that for all $\eta$ in a large class of
natural initial distributions (such as those having densities with
respect to Lebesgue measure), we have $\ra(\eta)=\ra_*$ and
$\eta^{(\infty)}=\mu_*$. In such a situation, $\mu_*$ can be thought
of as {\it a physical measure for the leaky system $(f,M,H)$}, in
analogy with the idea of physical measures for closed systems.

A Borel probability measure $\eta$ on $M$ is said to be
\emph{conditionally invariant} if it satisfies $\mathring
f_*\eta=\ra \eta$ for some $\ra \in (0,1]$. Clearly, the escape rate
of a conditionally invariant measure $\eta$ is well defined and is
equal to $- \log \ra$. Most leaky dynamical systems admit many
conditionally invariant measures; see [DY]. In particular, limiting
distributions, when they exist, are often conditionally invariant; they
are among the more important conditionally invariant measures from
an observational point of view.

Finally, when a physical measure $\eta$ for a leaky system $(f,M,H)$ has
absolutely continuous conditional measures on the
unstable manifolds of the underlying closed system $(f,M)$, we will
call it {\it an SRB measure for the leaky system},
in analogy with the idea of SRB measures for closed systems.

%%%%%%%%%%%%%%%%%%%%%%%%%%%%%%%%%

\subsection{Setting of present work}
\label{setting}

The underlying closed dynamical system here is the billiard map
associated with a 2-dimensional periodic Lorentz gas. Let $\{
\Gamma_i : i = 1, \cdots , d\}$ be pairwise disjoint $C^3$
simply-connected curves on ${\mathbb T}^2$ with strictly positive
curvature, and consider the billiard flow on the ``table" $X =
{\mathbb T}^2 \setminus \bigcup_i\{{\rm interior} \Gamma_i\}$. We
assume the ``finite horizon" condition, which imposes an upper bound
on the number of consecutive tangential collisions with $\cup
\Gamma_i$. The phase space of the unit-speed billiard flow is
$\mathcal{M} = (X \times {\mathbb S}^1)/\sim$ with suitable
identifications at the boundary. Let $M = \cup_i \Gamma_i \times
[-\frac{\pi}{2}, \frac{\pi}{2}] \subset \mathcal{M}$ be the
cross-section to the billiard flow corresponding to collision with
the scatterers, and let $f: M \to M$ be the Poincar\'e map. The
coordinates on $M$ are denoted by $(r,\varphi)$ where $r \in \cup
\Gamma_i$ is parametrized by arc length and $\varphi$ is the angle a
unit tangent vector at $r$ makes with the normal pointing into the
domain $X$. We denote by $\nu$ the invariant probability measure
induced on $M$ by Liouville measure on $\mathcal{M}$, {\it i.e.},
$d\nu = c \cos \varphi dr d\varphi$ where $c$ is the normalizing
constant.

We consider the following two types of holes:

\smallskip
\noindent {\it Holes of Type I.} In the table $X$, a hole $\sigma$
of this type is an open interval in the boundary of a scatterer.
When $q_0 \in \cup \Gamma_i$, we refer to $\{q_0\}$ as an {\it
infinitesimal hole}, and let $\Sigma_h(q_0)$ denote the collection
of all open intervals $\sigma \subset \cup \Gamma_i$ in the
$h$-neighborhood of $q_0$. A hole $\sigma$ in $X$ of this type corresponds to a
set $H_\sigma \subset M$ of the form $(a,b)\times [-\frac{\pi}{2},
\frac{\pi}{2}]$.

\smallskip
\noindent {\it Holes of Type II.} A hole $\sigma$ of this type is an
open convex subset of $X$ away from $\cup_i \Gamma_i$ and bounded by
a $C^3$ simple closed curve with strictly positive curvature. As
above, we regard $\{q_0 \}\subset X \setminus \cup \Gamma_i$ as an
infinitesimal hole, and use $\Sigma_h(q_0)$ to denote the set of all
$\sigma$ in the $h$-neighborhood of $q_0$. In this case, $\sigma
\subset X$ does not correspond directly to a set in $M$. Rather,
$\sigma$ corresponds directly to a set in $\mathcal{M}$, the phase
space for the billiard {\it flow}, and we must make a choice as to
which set in the cross section M will represent the hole for the
billiard map.  There is a well defined set $B_\sigma \subset M$
consisting of all $(r,\varphi)$ whose trajectories under the
billiard flow on $\mathcal{M}$ will enter $\sigma \times
\mathbb{S}^1$ before reaching $M$ again. Thus $H_\sigma =
f(B_\sigma)$ is a natural candidate for the hole in $M$ representing
$\sigma$, and will be taken as such in this work.  However, it would
also have been possible to take $B_\sigma$ as the representative
set.  The geometry of $B_\sigma$ and $H_\sigma$ in phase space will
be discussed in detail in Sect.~3.1. Also, we note that the
requirement that $\partial\sigma$ be a $C^3$ simple closed curve
with strictly positive curvature can be considerably relaxed. It is
even possible to allow some holes $\sigma$ that are not convex. See
the remark at the end of Sect. 3.1.

%%%%%%%%%%%%%%%%%%%%%%%%%%%%%%%%%%%

\subsection{Statement of results}
\label{results}

Let $\G = \G( H_\sigma)$ denote the set of finite Borel measures
$\eta$ on $M$ that are absolutely
continuous with respect to $\nu$ with $d\eta/d\nu$ being (i) Lipschitz on each connected component of $M$ and (ii) strictly positive on
$\cap_{i=0}^\infty f^{-i}\hM$.
 Notice that measures on $M$ with Lipschitz $d\eta/d\nu$
correspond to measures on
$\mathcal{M}$ having a Lipschitz density with respect to Liouville
measure.

\bigskip
\noindent {\bf Standing hypotheses for Theorems 1--3:} {\it We assume

(1) $f:M \to M$ is the billiard map defined in Sect.~1.2,

(2) $\{q_0\}$ is an infinitesimal hole of either Type I or Type II,
and

(3) $\sigma \in \Sigma_h(q_0)$ where $h>0$
is assumed to be sufficiently small. }

\begin{thm}{\bf (Common escape rate).}
\label{thm:escape rate}
All initial distributions $\eta \in \cal G$ have a common
escape rate $-\log \ra_*$ for some $\ra_*<1$; more precisely, for all
$\eta \in \cal G$, $\ra(\eta)$ is well
defined and is equal to $\ra_*$.
\end{thm}

\begin{thm}{\bf (Common limiting distribution).}
\label{thm:accim}

\noindent (a) For all $\eta \in \cal G$, the normalized surviving
distributions $\mathring f^n_*\eta/|\mathring f^n_*\eta|$ converge
weakly to a

common conditionally invariant distribution $\mu_*$ with
$\ra(\mu_*)=\ra_*$.

\noindent (b) In fact, for all $\eta \in \cal G$, there is a
constant $c(\eta) > 0$ s.t. $\ra_*^{-n} \mathring f^n_*\eta$
converges weakly to 

$c(\eta) \mu_*$.
\end{thm}

Thus from an observational point of view, $- \log \ra_*$ is the escape rate
and $\mu_*$ the physical measure for the leaky system
${(f,M,H_\sigma)}$.

\begin{thm}{\bf (Geometry of limiting distribution). }
\label{thm:geometry}

\noindent (a) $\mu_*$ is singular with respect to $\nu$;

\noindent  (b) $\mu_*$ has strictly positive conditional 
densities on local unstable manifolds.
\end{thm}

The precise meaning of the statement in part (b) of Theorem 3
 is that there are countably many
``patches" $(V_i, \mu_i)$, $i=1,2,\ldots$, where for each $i$,

(i) $V_i \subset M$ is the union of a continuous family of unstable curves
$\{ \gamma^u \}$;

(ii) $\mu_i$ is a measure on $V_i$ whose  conditional measures on
$\{\gamma^u\}$ have strictly positive

\quad densities with respect to the
Riemannian measures on $\gamma^u$;

(iii) $\mu_i \le \mu_*$ for each $i$, and $\sum_i \mu_i \geq \mu_*$.

\noindent This justifies viewing $\mu_*$ as the
SRB measure for the leaky system ${(f,M,H_\sigma)}$.

Our final result can be interpreted as a kind of stability for the natural
invariant measure $\nu$ of the billiard map without holes.

\begin{thm}{\bf (Small-hole limit).}
\label{thm:small hole limit} We assume (1) and (2) in the Standing
Hypotheses above. Let $\sigma_h \in \Sigma_h(q_0), h>0$, be an arbitrary
family of holes, and let  $-\log\ra_*(\sigma_h)$ and $\mu_*(\sigma_h)$
be the escape rate and physical measure for the leaky system
$(f,M,H_{\sigma_h})$. Then $\ra_*(\sigma_h) \to 1$ and
$\mu_*(\sigma_h) \to \nu$ as $h \to 0$.
\end{thm}

\medskip
\noindent {\bf Some straightforward generalizations:}
Our proofs continue to hold under the more general conditions below, but
we have elected not to discuss them (or to include them formally in the
statement of our theorems) because keeping track of an increased
number of objects will necessitate more cumbersome notation.

\medskip
\noindent 1. {\it Holes.} Our results apply to more general classes
of holes than those described above. For example,
we could fix a finite number of infinitesimal holes $\{q_0\},
..., \{q_k\}$ and consider $\sigma = \cup_i \sigma_i$ with $\sigma_i
\in \Sigma_h(q_i)$. In fact, we may take more than one $\sigma_i$ in
each $\Sigma_h(q_i)$ for as long as the total number of holes is
uniformly bounded.  See Sect.~3.4 for further generalizations on the types
of holes allowed.

\medskip
\noindent 2. {\it Initial distributions.} Theorems 1 and 2 (and consequently
Theorems 3 and 4) remain true with
$\mathcal{G}$ replaced by a broader class of measures.
For example, we use only the Lipschitz property of $d\eta/d\nu$
along unstable leaves, and it is sufficient
for $d\eta/d\nu$ to be strictly positive on large enough open sets
(see Remark~\ref{rem:generalize}). Moreover, $d\eta/d\nu$ need not be bounded
provided it blows up
sufficiently slowly near the singularity set for $f$.
Finally, we remark that Theorem 2(b) continues to hold without requiring
that $d \eta / d \nu$ be strictly positive anywhere, except that
now $c(\eta)$ might be $0$.

%%%%%%%%%%%%%%%%%%%%%%%%%%%%%%%%%%%%%%
%%%%%%%%%%%%%%%%%%%%%%%%%%%%%%%%%%%%%
\section{Relevant Dynamical Structures}
\label{relevant structures}

Our plan is to show that the billiard maps described in Sect.~1.2
admit certain structures called ``generalized horseshoes" which can
be arranged to ``respect the holes." The main results are summarized
in Proposition 2.2 in Sect. 2.2 and proved in Sect.~3.

\subsection{Generalized horseshoes}
\label{general horseshoe}

We begin by recalling the idea of a {\it horseshoe with infinitely
many branches and variable return times} introduced in [Y] for
general dynamical systems without holes. These objects will be
referred to in this paper as ``generalized horseshoes".

Following the notation in Sect.~1.1 of [Y], we consider a smooth or
piecewise smooth invertible map $f: M \to M$, and let $\mu$ and $\mu_\gamma$
denote respectively the Riemannian measure on $M$ and on $\gamma$
where $\gamma \subset M$ is a submanifold. We say the pair
$(\Lambda, R)$ defines a {\it generalized horseshoe} if {\bf
(P1)}--{\bf (P5)} below hold (see [Y] for precise formulation):

\begin{itemize}
\item[{\bf (P1)}] $\Lambda$ is a compact subset of $M$ with a hyperbolic product
structure, {\it i.e.}, $\Lambda = (\cup \Gamma^u) \cap (\cup
\Gamma^s)$ where $\Gamma^s$ and $\Gamma^u$ are continuous families
of local stable and unstable manifolds, and $\mu_{\gamma}\{\gamma
\cap \Lambda\}>0$ for every $\gamma \in \Gamma^u$.
\item[{\bf (P2)}] $R: \Lambda \to {\mathbb Z}^+$ is a {\it return time function}
to $\Lambda$. Modulo a set of $\mu$-measure zero, $\Lambda$ is the
disjoint union of $s$-subsets $\Lambda_j, j=1,2, \cdots,$ with the
property that for each $j$, $R|_{\Lambda_j}=R_j \in {\mathbb Z}^+$
and $f^{R_j}(\Lambda_j)$ is a $u$-subset of $\Lambda$.
\end{itemize}
There is a notion of {\it separation time} $s_0(\cdot, \cdot)$,
depending only on the unstable coordinate, defined for pairs of
points in $\Lambda$, and there are numbers $C>0$ and $\alpha <1$
such that the following hold for all $x,y \in \Lambda$:
\begin{itemize}
\item[{\bf (P3)}] For $y \in \gamma^s(x)$, $d(f^nx,f^ny) \le C\alpha^n$
for all $n \ge 0$.
\item[{\bf (P4)}] For $y \in \gamma^u(x)$ and $0 \le k \le n < s_0(x,y)$,

(a) $d(f^nx,f^ny) \le C \alpha^{s_0(x,y)-n}$;

(b) $\log \Pi_{i=k}^n \frac{\det Df^u(f^ix)}{\det Df^u(f^iy)} \ \le
\ C \alpha^{s_0(x,y)-n}.$

\item[{\bf (P5)}] (a) For $y \in \gamma^s(x)$,
$\log \Pi_{i=n}^\infty \frac{\det Df^u(f^ix)}{\det Df^u(f^iy)} \ \le
\ C \alpha^n$ for all $n \ge 0.$

(b) For $\gamma, \gamma' \in \Gamma^u$, if $\Theta : \gamma \cap
\Lambda \to \gamma' \cap \Lambda$ is defined by
$\Theta(x)=\gamma^s(x) \cap \gamma'$, then $\Theta$ is absolutely
continuous and $\frac{d(\Theta_*^{-1}\mu_{\gamma'})}{d\mu_\gamma}(x)
\ = \ \Pi_{i=0}^\infty \frac{\det Df^u(f^ix)}{\det Df^u(f^i\Theta
x)}$.
\end{itemize}

The meanings of the last three conditions are as follows: Orbits
that have not ``separated" are related by local hyperbolic
estimates; they also have comparable derivatives. Specifically, {\bf
(P3)} and {\bf (P4)}(a) are (nonuniform) hyperbolic conditions on
orbits starting from $\Lambda$. {\bf (P4)}(b) and {\bf (P5)} treat
more refined properties such as distortion and absolute continuity
of $\Gamma^s$, conditions that are known to hold for
$C^{1+\varepsilon}$ hyperbolic systems.

We say the generalized horseshoe $(\Lambda, R)$ has {\it exponential
return times} if there exist $C_0>0$ and $\theta_0>0$ such that for
all $\gamma \in \Gamma^u$, $\mu_\gamma\{R>n\} \le C_0 \theta_0^n$
for all $n \ge 0$.

\medskip
The setting described above is that of [Y]; it does not involve
holes. In this setting, we now identify a set $H \subset M$ (to be
regarded later as the hole) and introduce a few relevant
terminologies. Let $(\Lambda, R)$ be a generalized horseshoe for $f$
with $\Lambda \subset (M \setminus H)$.

We say $(\Lambda, R)$ {\it respects} $H$ if for every $i$ and every
$\ell$ with $0 \le \ell \le R_i$, $f^\ell(\Lambda_i)$ either does
not intersect $H$ or is completely contained in $H$.

The following definitions of ``mixing" are motivated by Markov-chain
considerations: Let  $\Lambda^s \subset \Lambda$ be an $s$-subset.
We say $\Lambda^s$ makes a {\it full return} to $\Lambda$ at time
$n$ if there are numbers $i_0, i_1, \cdots, i_k$ with $n=R_{i_0}+
\cdots + R_{i_k}$ such that $\Lambda^s \subset \Lambda_{i_0},
f^{R_{i_0} + \cdots + R_{i_j}}(\Lambda^s) \subset \Lambda_{i_{j+1}}$
for $j < k$, and $f^n(\Lambda^s)$ is a $u$-subset of $\Lambda$. (i)
We say the horseshoe $(\Lambda, R)$ is {\it mixing} if there exists
$N$ such that for every $n \ge N$, some $s$-subset $\Lambda^s(n)$
makes a full return at time $n$. (ii) If $(\Lambda, R)$ respects
$H$, then when we treat $H$ as a hole, we say the {\it surviving
dynamics} are {\it mixing} if in addition to the condition in (i),
we require that $f^\ell\Lambda^s(n) \cap H = \emptyset$ for all
$\ell$ with $0\le \ell \le n$. This is equivalent to requiring that
$\Lambda^s(n)$ makes a full return to $\Lambda$ at time $n$ under
the dynamics of $\mathring f$, where $\mathring f$ is the map
defined in Sect. 1.1.

We note that the mixing of $f$ in the usual sense of ergodic theory
does not imply that any generalized horseshoe constructed is
necessarily mixing in the sense of the last paragraph, nor does
mixing of the horseshoe imply that of its surviving dynamics.

%%%%%%%%%%%%%%%%%%%%%%%%%%%%%%%%%%%%

\subsection{Main Proposition for billiards with holes}

With these general ideas out of the way, we now return to the
setting of the present paper. From here on, $f: M \to M$ is the
billiard map of the 2-D Lorentz gas as in Sect.~1.2. The following
result lies at the heart of the approach taken in this paper:

\begin{proposition} {\rm (Theorem 6(a) of [Y])} The map
$f$ admits a generalized horseshoe with exponential return times.
\end{proposition}

A few more definitions are needed before we are equipped to state
our main proposition: We call $Q \subset M$ a rectangular region if
$\partial Q = \partial^u Q \cup \partial^s Q$ where $\partial^u Q$
consists of two unstable curves and $\partial^s Q$ two stable
curves. We let $Q(\Lambda)$ denote the smallest rectangular region
containing $\Lambda$, and define $\mu^u(\Lambda):= \inf_{\gamma \in
\Gamma^u} \mu_\gamma(\Lambda \cap \gamma)$. Finally, for a
generalized horseshoe $(\Lambda, R)$ respecting a hole $H$, we
define
$$n(\Lambda, R; H)= \sup \{n \in {\mathbb Z}^+: \
{\rm no \ point \ in} \ \Lambda \ {\rm falls \ into} \ H \ {\rm in \
the\ first} \ n \ \rm iterates\}.
$$

In the rest of this paper, $C$ and $\alpha$ will be the constants in
{\bf (P3)}--{\bf (P5)} for the closed system $f$. All notation is as
in Sect.~1.2.

\begin{proposition}
\label{prop:horseshoe} Given an infinitesimal hole $\{q_0\}$ of Type
I or II, there exist $C_0, \kappa>0$, $\theta_0 \in (0,1)$, and a
rectangular region $Q$ such that for all small enough $h$ we have
the following:
\begin{itemize}
\item[(a)] For each $\sigma \in \Sigma_{h}(q_0)$,

(i) $f$ admits a generalized horseshoe $(\Lambda^{(\sigma)},
R^{(\sigma)})$ respecting $H_\sigma$;

(ii) both $(\Lambda^{(\sigma)}, R^{(\sigma)})$ and the corresponding
 surviving dynamics are mixing.
\item[(b)] All $\sigma \in \Sigma_{h}(q_0)$ have the following uniform properties:

(i) $Q(\Lambda^{(\sigma)}) \approx Q$ \footnote{By
$Q(\Lambda^{(\sigma)}) \approx Q$, we only wish to convey that both
rectangular regions are located in roughly the same region of the
phase space, $M$, and not anything technical in the sense of
convergence.} , and $\mu^u(\Lambda^{(\sigma)}) \ge \kappa$;

(ii) $\mu_\gamma\{R^{(\sigma)}>n\} < C_0 \theta_0^n$ for all $n \ge
0$;

(iii) (P3)--(P5) hold with the constants $C$ and $\alpha$.
\end{itemize}
Moreover, if $\bar n(h)= \inf_{\sigma \in \Sigma_h(q_0)} n(\Lambda,
R; H_\sigma)$, then $\bar n(h) \to \infty$ as $h \to 0$.
\end{proposition}

 \noindent {\bf Clarification:}

\smallskip \noindent
1. Here and in Sect. 3, there is a set, namely $H_\sigma$, that is
identified to be ``the hole," and a horseshoe is constructed to
respect it. Notice that the construction is continued after a set
enters $H_\sigma$. For reasons to become clear in Sect.~6, we cannot
simply disregard those parts of the phase space that lie in the
forward images of $H_\sigma$.

\smallskip \noindent
2. Proposition 2.2 treats only small $h$, i.e. small holes. The
smallness of the holes and the uniformness of the estimates in part
(b) are needed for the spectral arguments in Sect.~4 to apply.
Without any restriction on $h$, all the conclusions of
Proposition 2.2 remain true except for the following: (a)(ii), where
for large holes the surviving dynamics need not be mixing, (b)(i),
and (b)(ii), where $C_0$ and $\theta_0$ may be $\sigma$-dependent.
The assertions for large $h$ will be evident from our proofs; no separate
arguments will be provided.

\bigskip
A proof of Proposition 2.2 will require that we repeat the
construction in the proof of Proposition 2.1 -- and along the way,
to carry out a treatment of holes and related issues. We believe it
is more illuminating conceptually (and more efficient in terms of
journal pages) to focus on what is new rather than to provide a
proof written from scratch. We will, therefore, proceed as follows:
The rest of this section contains a review of all the arguments used
in the proof of Proposition 2.1, with technical estimates omitted
and specific references given in their place. A proof of Proposition
2.2 is given in Sect.~3. There we go through the same arguments point
by point, explain where modifications are needed and treat new
issues that arise. For readers willing to skip more technical
aspects of the analysis not related to holes, we expect that they
will get a clear idea of the proof from this paper alone. For
readers who wish to see all detail, we ask that they read this proof
alongside the papers referenced.

%%%%%%%%%%%%%%%%%%%%%%%%%%%%%%%%%

\subsection{Outline of construction in [Y]}
\label{young horseshoe}

In this subsection, the setting and notation are both identical to
that in Sect.~8 of [Y]. Referring the reader to [Y] for detail, we
identify below 7 main ideas that form the crux of the proof of
Proposition 2.1. We will point out the use of billiard properties
and other geometric facts that may potentially be impacted by the
presence of holes. Holes are not discussed explicitly, however,
until Sect.~3.

\medskip
\noindent {\bf Notation and conventions:}  In [Y], $S_0$ and
$\partial  M$ were used interchangeably. Here we use exclusively
$\partial M$. Clearly, $f^{-1}\partial M$ is the discontinuity set of $f$.

\smallskip
\noindent (i) {\it $u$- and $s$-curves.} Invariant cones $C^u$ and $C^s$
are fixed at each point, and curves all of whose tangent vectors are
in $C^u$ (resp. $C^s$) are called $u$-curves (resp. $s$-curves).

\smallskip
\noindent (ii) {\it The $p$-metric.} Euclidean distance on $M$ is
denoted by $d(\cdot, \cdot)$. Unless declared otherwise, distances
and derivatives along $u$- and $s$-curves are measured with respect
to a semi-metric called the $p$-metric defined by $\cos\varphi dr$.
These two metrics are related by $cp(x,y) \le d(x,y) \le p(x,y)^{\frac12}$.
By $W^u_\delta(x)$, we refer to the piece of local unstable curve of
$p$-length $2\delta$ centered at $x$. {\bf (P3)}--{\bf (P5)} in Sect. 2.1
hold with respect to the $p$-metric.  See Sect. 8.3 in [Y] for details.

\smallskip
\noindent (iii) {\it Derivative bounds.} With respect to the
$p$-metric, there is a number $\lambda>1$ so that all vectors in
$C^u$ are expanded by $\ge \lambda$ and all vectors in $C^s$
contracted by $\le \lambda^{-1}$. Furthermore, derivatives at $x$
along $u$-curves are $\sim d(x,
\partial M)^{-1}$. For purposes of distortion control, homogeneity
strips of the form
$$
I_k \ = \ \left \{(r,\varphi): \frac{\pi}{2} - \frac{1}{k^2} <
\varphi < \frac{\pi}{2} - \frac{1}{(k+1)^2} \right\}, \quad k \ge
k_0,
$$
are used, with $\{I_{-k}\}$ defined similarly in a neighborhood of
$\varphi = -\frac{\pi}{2}$. For convenience, we will refer to $M
\setminus (\cup_{|k| \ge k_0} I_k)$ as one of the ``$I_k$".

\medskip
\noindent {\bf Important Geometric Facts ($\dagger$):} The following
facts are used many times in the proof:
\begin{enumerate}   \vspace{-8 pt}
  \item[(a)]  the discontinuity set $f^{-1}\partial M$ is the union
of a finite number of compact piecewise smooth decreasing
curves, each of which stretches from $\{ \varphi = \pi / 2 \}$ to
$\{ \varphi = -\pi / 2 \}$;  \vspace{-8 pt}
  \item[(b)] $u$-curves are uniformly transversal (with angles
bounded away from zero) to $\partial M$ and to $f^{-1}\partial M$.
\end{enumerate}

\noindent {\bf 1. Local stable and unstable manifolds.} Only {\it
homogeneous} local stable and unstable curves are considered.
Homogeneity for $W^u_\delta$, for example, means that for all $n \ge
0$, $f^{-n}W^u_\delta$ lies in no more than 3 contiguous $I_k$. Let
$\delta_1>0$ be a small number to be chosen. We let $\lambda_1
=\lambda^{\frac14}$,  $\delta = \delta_1^4$, and define
$$
B^+_{\lambda_1, \delta_1} = \{x \in M: d(f^nx, \partial M \cup
f^{-1}(\partial M)) \ge \delta_1 \lambda_1^{-n} \ {\rm \ for \ all}
\ n \ge 0\}\ ,
$$
$$
B^-_{\lambda_1, \delta_1} = \{x \in M: d(f^{-n}x, \partial M \cup
f(\partial M)) \ge \delta_1 \lambda_1^{-n} \ {\rm \ for \ all} \ n
\ge 0\}\ .
$$
We require $d(f^nx, f^{-1}(\partial M))\ge \delta_1 \lambda_1^{-n}$
to ensure the existence of a local unstable curve through
$x$, while the requirement on $d(f^nx,
\partial M)$ is to ensure its homogeneity.\footnote{ In
fact, provided $\delta_1$ is chosen sufficiently small,
one can verify that $d(f^n x, f^{-1} (\partial M)) \ge \delta_1 \lambda_1^{-n}$
implies that $d(f^{n + 1} x, \partial M) \ge \delta_1 \lambda_1^{-(n + 1)}$
for all $n \ge 0$.  This fact, which was not used in \cite{young},
will be used in item 2 below to simplify our presentation. }
Similar reasons apply to
stable curves.  Observe that (i) for all $x \in B^+_{\lambda_1,
\delta_1}$, $W^s_{10 \delta}(x)$ is well defined and homogeneous
(this is straightforward since $\delta << \delta_1$ and $\lambda_1$
is closer to $1$ than $\lambda$); and (ii) as $\delta_1 \to 0$,
$\nu(B^+_{\lambda_1, \delta_1}) \to 1$ (this follows from a standard
Borel-Cantelli type argument). Analogous statements hold for
$B^-_{\lambda_1, \delta_1}$.

\medskip
\noindent {\bf 2. Construction of the Cantor set $\Lambda$.} The
choice of $\Lambda$ is, in fact, quite arbitrary. We pick a density
point $x_1$ of $B^+_{\lambda_1, 2\delta_1} \cap B^-_{\lambda_1,
2\delta_1}$ at least $2\delta_1$ away from $f^{-1}(\partial M) \cup
\partial M \cup f(\partial M)$, and let $\Omega =
W^u_\delta(x_1)$.\footnote{Later we will impose one further
technical condition on the choice of $x_1$.  See the very end of
Sect.~2.4.} For each $n$, we define
$$
\Omega_n = \{y \in \Omega : d(f^iy, f^{-1}(\partial  M) ) \geq
\delta_1 \lambda_1^{-i}\ \text{ for } 0 \leq i \leq n\}\ ,
$$
and let $\Omega_\infty = \cap_n \Omega_n$. Then
$\Omega_\infty\subset B^+_{\lambda_1, \delta_1}$, by the
footnote in item 1 above and our choice of $x_1$ far from $\partial M$.
Let $\Gamma^s$
consist of all $W^s_\delta(y), y \in \Omega_\infty$, and let
$\Gamma^u$ be the set of all homogeneous $W^u_{\text{loc}}$ curves
that meet every $\gamma^s \in \Gamma^s$ and which extend by a
distance $> \delta$ on both sides of the curves in $\Gamma^s$. The
set $\Lambda$, which is defined to be $(\cup \Gamma^u) \cap (\cup
\Gamma^s)$, clearly has a hyperbolic product structure. {\bf
(P5)(b)} is standard. This together with the choice of $x_1$
guarantees $\mu_\gamma\{\gamma \cap \Lambda\}>0$ for $\gamma \in
\Gamma^u$, completing the proof of {\bf (P1)}.

A natural definition of separation time for $x,y \in \gamma^u$ is as
follows: Let $[x,y]$ be the subsegment of $\gamma^u$ connecting $x$
and $y$. Then $f^nx$ and $f^ny$ are ``not yet separated," i.e.
$s_0(x,y) \geq n$, if  for all $i \le n$, $f^i[x,y]$  is connected
and is contained in at most 3 contiguous $I_k$. With this definition
of $s_0(\cdot, \cdot)$, {\bf (P3)}--{\bf (P5)(a)} are checked using
previously known billiard estimates.

\medskip
\noindent {\bf 3. The return map $f^R: \Lambda \to \Lambda$.} We
point out that there is some flexibility in choosing the return map
$f^R$: Certain conditions have to be met when a return takes place,
but when these conditions are met, we are not obligated to call it a
return; in particular, $R$ is not necessarily the {\it first} time
an $s$-subrectangle of $Q$ $u$-crosses $Q$ where $Q=Q(\Lambda)$.

We first define $f^R$ on $\Omega_\infty$.  Let $\tilde \Omega_n =
\Omega_n \setminus \{R \le n\}$. On  $\tilde \Omega_n$ is a
partition $\tilde {\cal P}_n$ whose elements are segments
representing distinct trajectories. The rules are different before
and after a certain time $R_1$, a lower bound for which is
determined by $\lambda_1$, $\delta_1$ and the derivative of
$f$.\footnote{In [Y], properties of $R_1$ are used in 4 places:
(I)(i) in Sect.~3.2, Sublemma 3 in Sect.~7.3, the paragraph
following (**) in Sect.~8.4., and a requirement in Sect.~8.3 that
stable manifolds pushed forward more than $R_1$ times are
sufficiently contracted.}

 (a) For $n <R_1$, $\tilde {\cal P}_n$ is constructed from
the results of the previous step\footnote{ In [Y], it was sufficient
to allow returns to $\Lambda$ at times that were multiples of a
large fixed integer $m$. Not only is this not necessary (see
Paragraph 4), here it is essential that we avoid such periodic
behavior to ensure mixing. Thus we take $m=1$ when choosing return
times in Paragraph 3. This is the only substantial departure we make
from the construction in [Y].  } as follows: Let $\omega \in \tilde
{\cal P}_{n-1}$, and let $\omega'$ be a component of $\omega \cap
\Omega_n$. Inserting cut-points only where necessary, we divide
$\omega'$ into subsegments $\omega_i$ with the property that
$f^n(\omega_i)$ is homogeneous. These are the elements of $\tilde
{\cal P}_n$. No point returns before time $R_1$.

(b) For $n \ge R_1$, we proceed as in (a) to obtain $\omega_i$. If
$f^n(\omega_i)$ $u$-crosses the middle of $Q$ with $ \ge 1.5\delta$
sticking out on each side, then we declare that $R=n$ on $\omega_i
\cap f^{-n} \Lambda$, and the elements of $\tilde {\cal
P}_n|_{\omega_i \cap \tilde \Omega_n}$ are the connected components
of $\omega_i \setminus f^{-n} \Lambda$. Otherwise put  $\omega_i \in
\tilde {\cal P}_n$ as before.

This defines $R$ on a subset of $\Omega_\infty$ (which we do not
know yet has full measure); the definition is extended to the
associated $s$-subset of $\Lambda$ by making $R$ constant on
$W^s_{\text{loc}}$-curves. The $s$-subsets associated with $\omega_i
\cap f^{-n} \Lambda$ in (b) above are the $\Lambda_j$ in {\bf (P2)}.
It remains to check that $f^R(\Lambda_j)$ is in fact a $u$-subset of
$\Lambda$. This is called the ``matching of Cantor sets" in [Y] and
is a consequence of the fact that $\Omega_\infty$ is dynamically
defined and that $R_1$ is chosen sufficiently large.

\medskip
It remains to prove that $p\{R \ge n\}$ decays exponentially with
$n$. Paragraphs 4, 5 and 6  contain the 3 main ingredients of the
proof, with the final count given in 7.

\medskip
\noindent {\bf 4. Growth of $u$-curves to ``long" segments.} This is
probably the single most important point, so we include a few more
details. We first give the main idea before adapting it to the form
it is used. Let $\eps_0>0$ be a number the significance of which we
will explain later. Here we think of a $u$-curve whose $p$-length
exceeds $\eps_0>0$ as ``long". Consider a $u$-curve $\omega$. We
introduce a stopping time $T$ on $\omega$ as follows. For $n=1,2,
\cdots$, we divide $f^n\omega$ into homogeneous segments
representing distinguishable trajectories. For $x \in \omega$, let
$$
T(x) = \inf \{n>0 :{\rm the \ segment \ of} \ f^n\omega \ {\rm
containing} \ f^nx \ {\rm has} \ p{\rm -length} \ > \eps_0\}\ .
$$

\begin{lemma} There exist $D_1>0$ and $\theta_1<1$ such that for any $u$-curve
$\omega$,
$$
p(\omega \setminus \{T \le n\}) < D_1 \theta_1^n \qquad {\rm for \
all} \ n \ge 1.
$$
\end{lemma}

This lemma relies on the following important geometric property of
the class of billiards in question. This choice of $\eps_0>0$ is
closely connected to this property:

\begin{itemize}
\item[(*)]  ([BSC1], Lemma 8.4)
 {\it The number of curves in $\cup_{i=1}^n f^{-i} (\partial M)$ passing through or ending
 in any one point in $M$
is $\leq K_0 n$, where $K_0$ is a constant depending only on the
``table" $X$.}
\end{itemize}

Let $\alpha_0 := 2 \sum_{ k=k_0}^\infty \frac{1}{k^2}$ where $\{I_k,
|k| \geq k_0\}$ are the homogeneity strips, and assume that
$\lambda^{-1} + \alpha_0 < 1$. Choose $m$ large enough that
$\theta_1:= (K_0 m+1)^{\frac{1}{m}}(\lambda^{-1} + \alpha_0) < 1$.
We may then fix $\eps_0 < \delta$ to be small enough that every
$W^u_{\text{loc}}$-curve of $p$-length $\leq \eps_0$ has the
property that it intersects $\leq K_0 m$ smooth segments of
$\cup_1^m f^{-i} (\partial M)$, so that the $f^m $-image of such a
$W^u_{\text{loc}}$-curve has $\le (K_0m+1)$ connected components.

The proof of Lemma 2.3, which follows [BSC2], goes as follows:
Consider a large $n$, which we may assume is a multiple of $m$.
(Once Lemma 2.3 is proved for multiples of $m$, the estimate can be
extended to intermediate values by enlarging the constant $D_1$.) We
label distinguishable trajectories by their $I_k$-itineraries.
Notice that because $f^i\omega$ is the union of a number of
(disconnected) $u$-curves, it is possible for many distinguishable
trajectories to have the same $I_k$-itinerary. Specifically, by (*),
each trajectory of length $jm, j \in {\mathbb Z}^+$, gives birth to
at most $(K_0 m+1)$ trajectories of length $(j+1)m$ with the same
$I_k$-itinerary. To estimate $p(\omega \setminus \{T \le n\})$, we
assume the worst case scenario, in which the $f^n$-images of
subsegments of $\omega$ corresponding to {\it all} distinguishable
trajectories have length $\le \eps_0$. We then sum over all possible
itineraries using bounds on $Df$ along $u$-curves in $I_k$.

We now adapt Lemma 2.3 to the form in which it will be used. Let
$\omega = f^k\omega'$ for some $\omega' \in \tilde {\cal P}_k$ in
the construction in Paragraph 3. As we continue to evolve $\omega$,
$f^n\omega$ is not just chopped up by the discontinuity set, bits of
it that go near $f^{-1}(\partial M)$ will be lost by intersecting
with $f^{k+n} \Omega_{k+n}$, and we need to estimate $p(\omega_n
\setminus \{T \le n\})$ where $\omega_n := \omega \cap
f^k(\Omega_{k+n})$ takes into consideration these intersections and
$T$ is redefined accordingly. {\it A priori} this may require a
larger bound than that given in Lemma 2.3: it is conceivable that
there are segments that will grow to length $\eps_0$ without losing
these ``bits" but which do not now reach this reference length. We
claim that all such segments have been counted, because (i) the
deletion procedure does not create new connected components; it
merely trims the ends of segments adjacent to cut-points; and (ii)
the combinatorics in Lemma 2.1 count all possible itineraries (and
not just those that lead to ``short" segments). This yields the
desired estimate on $p(\omega_n \setminus \{ T \le n \})$, which is
Sublemma 2 in Sect. 8.4 of [Y].

\medskip
\noindent {\bf 5. Growth of ``gaps" of $\Lambda$.} Let $\omega$ be
the subsegment of some $\gamma^u \in \Gamma^u$ connecting the two
$s$-boundaries of $Q$. We think of this as a return in the
construction outlined in Paragraph 3, with the connected components
$\omega'$ of $\omega^c= \omega \setminus \Lambda$ being $f^k$-images
of elements of $\tilde {\cal P}_k$. We define a stopping time $T$ on
$\omega^c$ by considering one $\omega'$ at a time and defining on it
the stopping time in Paragraph 4.

\begin{lemma} There exist $D_2>0$ and $\theta_2<1$ independent of $\omega$
such that
$$
p(\omega_n^c \setminus \{T \le n\}) < D_2 \theta_2^n \qquad {\rm for
\ all} \ n \ge 1.
$$
\end{lemma}

The idea of the proof is as follows. We may identify $\omega$ with
$\Omega$ (see Paragraph 2), so that the collection of $\omega'$ is
precisely the collection of gaps in $\Lambda$. We say $\omega'$ is
of generation $q$ if this is the first time a part of $\omega'$ is
removed in the construction of $\Omega_\infty$. There are two
separate estimates:
$$(I) := \sum_{q > \eps n} \sum_{\text{gen}(\omega') = q} p(\omega');
\qquad (II):= \sum_{q \leq \eps n} \sum_{\text{gen}(\omega')= q}
p(\omega_n' \backslash \{T \leq n\}). $$ (I)  has exponentially
small $p$-measure: this follows from a comparison of the growth rate
of $Df$ along $u$-curves {\it versus} the rate at which these curves
get cut (see Paragraph 4). (II) is bounded above by
$$
\sum_{q \le \varepsilon n} \ \sum_{{\rm gen}(\omega')=q}
\frac{Cp(\omega')}{p(f^{q-1}\omega')} \cdot D_1 \theta_1^{n-q-1}\ .
$$
This is obtained by applying the modified version of Lemma 2.3 to
$f^{q-1}\omega'$. A lower bound on $p(f^{q-1}\omega')$ can be
estimated as these curves have not been cut by $f^{-1}(\partial M)$
(though they may have been shortened to maintain homogeneity),
reducing the estimate to $\sum_q \sum_{{\rm gen}(\omega')=q}
p(\omega')$, which is $\le p(\omega)$.

\medskip
\noindent {\bf 6. Return of ``long" segments.} This concerns the
evolution of unstable curves after they have grown ``long", where
``long" has the same meaning as in Paragraph 4. The following
geometric fact from [BSC2] is used:

\medskip
 (**) \ \ {\it Given $\varepsilon_0 > 0, \ \exists n_0$ s.t. for every
homogeneous $W^u_{\text{loc}}$-curve $\omega$ with $p(\omega)
> \varepsilon_0$ and every $q \ge n_0$, $f^q \omega$ contains a
homogeneous segment which $u$-crosses the middle half of $Q$ with $>
2\delta$ sticking out from each side.}

\medskip
We choose $\varepsilon_0 > 0$ as explained in Paragraph 4 above, and
apply (**) with $q =n_0 $ to the segments that arise in Paragraphs 4
and 5
 when the stopping time $T$ is
reached. For example, $\omega$ here may be equal to $f^n \omega''$
where $\omega''$ is a subsegment of the $\omega$ in the last
paragraph of Paragraph 4 with $T|_{\omega''}=n$. We claim that a
fixed fraction of such a segment will make a return within $n_0$
iterates. To guarantee that, two other facts need to be established:
(i) The small bits deleted by intersecting with
$f^{n+k}\Omega_{n+k}$ before the return still leave a segment which
$u$-crosses the middle half of $Q$ with $> 1.5\delta$ sticking out
from each side; this is easily checked. (ii) For $q \le n_0$,
$(f^q)'$ is uniformly bounded on $f^{-q}$-images of homogeneous
segments that $u$-cross $Q$. This is true because a segment
contained in $I_k$ for too large a $k$ cannot grow to length
$\delta$ in $n_0$ iterates.

\medskip
\noindent {\bf 7. Tail estimate of return time.} We now prove $p\{R
\ge n\} \le C_0 \theta_0^n$ for some $\theta_0<1$. On $\Omega$,
introduce a sequence of stopping times $T_1<T_2 < \cdots$ as
follows: A stopping time $T$ of the type in Paragraph 4 or 5 is
initiated on a segment as soon as $T_k$ is reached, and $T_{k+1}$ is
set equal to $T_k+T$. In this process, we stop considering points
that are lost to deletions or have returned to $\Lambda$. The
desired bound follows immediately from the following two estimates:
\begin{itemize}
\item[(i)] There exists $\varepsilon'>0, D_3 \ge 1$, and $\theta_3<1$ such that
$p(T_{[\varepsilon' n]}>n) < D_3 \theta_3^n$.
\item[(ii)] There exists $\varepsilon_1>0$ such that if $T_k|\omega = n$, then
$p(\omega \cap \{R > n+n_0\}) \le (1-\varepsilon_1) p(\omega)$ where
$n_0$ is as in (**) in Paragraph 6.
 \end{itemize}
(ii) is explained in Paragraph 6. To prove (i), we let
$p=[\varepsilon' n]$, decompose $\Omega$ into sets of the form
$A(k_1, \cdots, k_p) =\{x \in \Omega: T_1(x), \cdots, T_p(x)$ are
defined with $T_i=k_i\}$, apply Lemmas 2.1 and 2.2 to each set and
recombine the results. The argument here is combinatorial, and does
not use further geometric information about the system.

%%%%%%%%%%%%%%%%%%%%%%%%%%%%%%

\subsection{Sketch of proof of (**) following [BSC2]}

Property (**) is a weaker version of Theorem 3.13 in [BSC2].  We
refer the reader to [BSC2] for detail, but include an outline of its
proof because a modified version of the argument will be needed in
the proof of Proposition 2.2.

We omit the proof of the following elementary fact, which relies on
the geometry of the discontinuity set including Property (*):

\bigskip
\noindent {\bf Sublemma A.} {\it Given any $u$-curve $\gamma$,
through $\mu_\gamma$-a.e. $x \in \gamma$ passes a homogeneous
$W^s_{\delta(x)}(x)$ for some $\delta(x)>0$. The analogous statement
holds for $s$-curves.}

\bigskip

Instead of considering every $W^u_{\text{loc}}$-curve  as required
in (**), the problem is reduced to a finite number of ``mixing
boxes" $U_1, U_2, \ldots, U_k $ with the following properties:

(i) $U_j$ is a hyperbolic product set defined by (homogeneous)
families $\Gamma^u(U_j)$ and $\Gamma^s(U_j)$;

\quad located in the middle third of $U_j$ is an $s$-subset $\tilde
U_j$ with $\nu(\tilde U_j)>0$;

(ii) $\cup \Gamma^u(U_j)$ fills up nearly $100\%$ of the measure of
$Q(U_j)$; and

(iii) every $W^u_{\text{loc}}$-curve $\omega$ with $p(\omega)>
\varepsilon_0$ passes through the middle third of one of the

\quad $Q(U_j)$ in the manner shown in  Fig.~1~(left).

\noindent That (i) and (ii) can be arranged follows from Sublemma A.
That a finite number of $U_j$ suffices for (iii) follows from a
compactness argument.

Next we choose a suitable subset $\tilde U_0 \subset \Lambda$ to be
used in the mixing. To do that, first pick a hyperbolic product set
$U_0$ related to $Q(\Lambda)$ as shown in Fig.~1~(right). We require
that it meet $Q(\Lambda)$ in
 a set of positive measure, that it
sticks out of $Q(\Lambda)$ in the $u$-direction by more than
$2\delta$, and that the curves in $\Gamma^u(U_0)$ fill up nearly
$100\%$ of $Q(U_0)$. Let $\ell_0>0$ be a small number, and let
$\tilde U_0 \subset U_0$ consist of those density points of $U_0\cap
Q(\Lambda)$ with the additional property that if a homogeneous
stable curve $\gamma^s$ with $p(\gamma^s)< \ell_0$ meets such a
point, then $p(\gamma^s \cap U_0)/p(\gamma^s) \approx 1$. For
$\ell_0$ small enough, $\nu(\tilde U_0)>0$ because the foliation
into $W_{\text{loc}}^{u}$-curves is absolutely continuous.

\begin{figure}[htp]
\centering
\resizebox{3 in}{!}{
\includegraphics{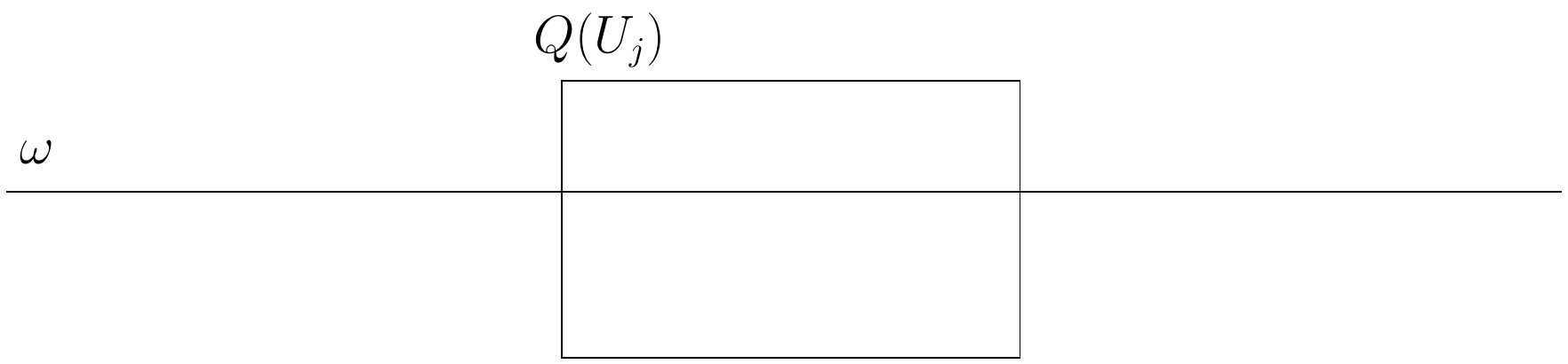}} \hspace{.4 in}
\resizebox{3 in}{!}{
\includegraphics{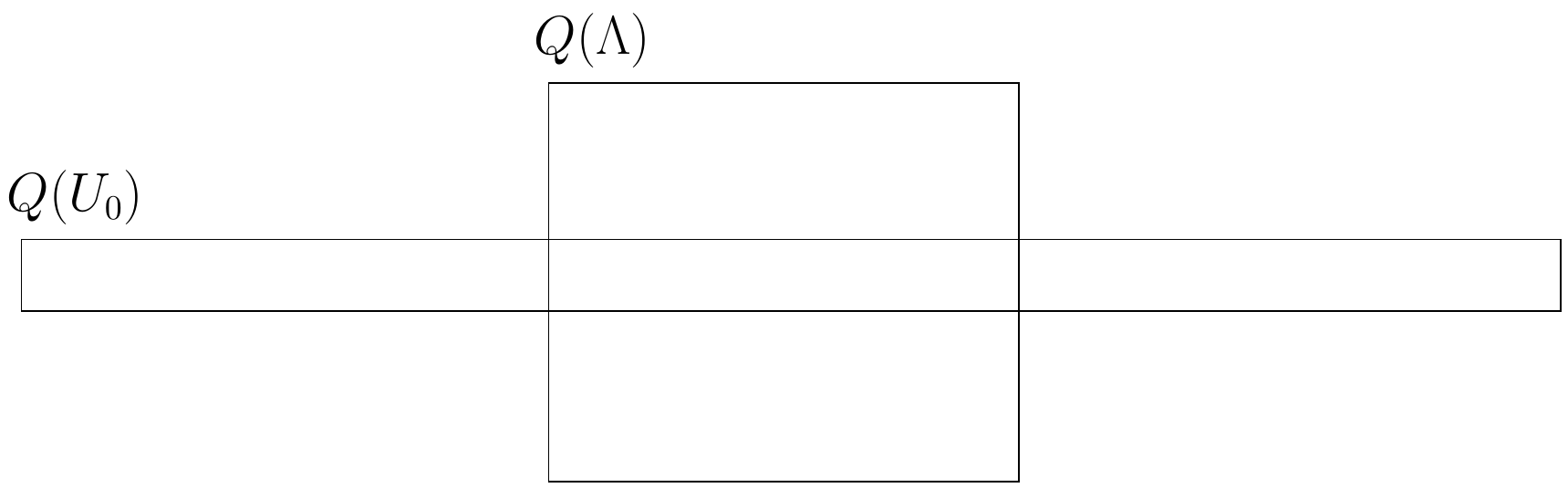}} \\
\label{figure 1}
\caption{{\it Left:} A mixing box $U_j$. {\it Right:}
The target box $U_0$.}
\end{figure}

By the mixing property of $(f,\nu)$, there exists $n_0$ such that
for all $q \ge n_0$, $\nu(f^q(\tilde U_j) \cap \tilde U_0)>0$ for
every $\tilde U_j$. We may assume also that $n_0$ is so large that
for $q \ge n_0$, if $x \in \tilde U_j$ is such that $f^qx \in \tilde
U_0$, then $p(f^q(\gamma^s(x)))< \ell_0$ where $\gamma^s(x)$ is the
stable curve in $\Gamma^s(\tilde U_j)$ passing through $x$. Let $q
\ge n_0$ and $j$ be fixed, and let $x \in \tilde U_j$ be as above.
From the high density of unstable curves in both $U_j$ and $U_0$, we
are guaranteed that there are two elements $\gamma^u_1, \gamma^u_2
\in \Gamma^u(U_j)$ sandwiching the middle third of $Q(U_j)$ such
that for each $i$, a subsegment of $\gamma^u_i$ containing
$\gamma^s(x) \cap \gamma^u_i$ is mapped under $f^q$ onto some $\hat
\gamma^u_i \in \Gamma^u(U_0)$. Let $Q^*=Q^*(q, j)$ be the
$u$-subrectangle of $Q(U_0)$ with $\partial^u Q^* = \hat \gamma^u_1
\cup \hat \gamma^u_2$.

\bigskip
\noindent {\bf Sublemma B.} $f^{-q}\mid_{Q^*}$ is continuous,
equivalently, $Q^*\cap (\cup_0^q f^{i}(\partial M)) = \emptyset$.

\bigskip
Sublemma B is an immediate consequence of the geometry of the
discontinuity set:  By the choice of $ x_1$ in item 2 of Sect.~2.3,
$Q^* \cap
\partial M = \emptyset$.  Suppose $Q^* \cap (\cup_1^q
f^{i}(\partial M)) \not= \emptyset$. Since $\cup_1^q f^{i}(\partial
M)$ is the union of finitely many piecewise smooth (increasing)
u-curves each connected component of which stretches from $\{
\varphi = -\pi/2 \}$ to $\{\varphi = \pi/2\}$, and these curves
cannot touch $\partial ^u Q^*$, a piecewise smooth segment from
$\cup_1^q f^{i}(\partial M)$ that enters $Q^*$ through one component
of $\partial^s Q^*$ must exit through the other. In particular, it
must cross $f^q\gamma^s(x)$, which is a contradiction.

To prove (**), let $\omega$ be a $W^u_{\text{loc}}$-curve with
$p(\omega)>\varepsilon_0$. We pick $U_j$ so that $\omega $ passes
through the middle third of $U_j$ as in (iii) above. Sublemma B then
guarantees that $f^q (\omega\cap f^{-q}Q^*)$ connects the two
components of $\partial^s Q^*$. This completes the proof of (**),
except that we have not yet verified that $f^q(\omega \cap
f^{-q}Q^*)$ is homogeneous.

To finish this last point, we modify the above argument as follows:
First, we define a $W_{\text{loc}}^u$ curve $\gamma$ to be {\it
strictly homogeneous} if for all $n\geq 0$, $f^{-n}\gamma$ is
contained inside one homogeneity strip $I_k(n)$.  Strict homogeneity
for $W_{\text{loc}}^s$ curves is defined analogously.  The
conclusions of Sublemma A remain valid if, in its statement, the
word ``homogeneous'' is replaced by ``strictly homogeneous.''  Thus
the mixing boxes $U_1, \ldots, U_k$ can be chosen so that their
defining families are comprised entirely of strictly homogeneous
local manifolds. Furthermore, if $x_1$ is also chosen as a density
point of points with sufficiently long strictly homogeneous unstable
curves, $\Gamma^u(U_0)$ can be chosen to be comprised entirely of
strictly homogeneous $W_{\text{loc}}^u$-curves.  Having done this,
an argument very similar to the proof of Sublemma B shows that
$f^{-i}Q^* \cap (\cup_k \partial I_k) = \emptyset $ for $0 \leq i
\leq q$, and this completes the proof of (**).

%%%%%%%%%%%%%%%%%%%%%%%%%%%%%%%%
%%%%%%%%%%%%%%%%%%%%%%%%%%%%%%%%

\section{Horseshoes Respecting Holes for Billiard Maps}
\label{horseshoe with hole}

\subsection{Geometry of holes in phase space}
\label{holes description}

We summarize here some relevant geometric properties and explain how
we plan to incorporate holes into our horseshoe construction.

\medskip
\noindent {\bf Holes of Type I}. Recall from Sect.~1.2 that for $q_0
\in \cup \Gamma_i$ and $\sigma \in \Sigma_h(q_0)$, $H_\sigma \subset
M$ is a rectangle of the form $(a,b) \times [-\frac{\pi}{2},
\frac{\pi}{2}]$. We define $\partial H_\sigma := \{a,b\} \times
[-\frac{\pi}{2}, \frac{\pi}{2}]$, i.e. $\partial H_\sigma$ is the
boundary of $H_\sigma$ viewed as a subset of $M$. It will also be
convenient to let $H_0 \subset M$ denote the vertical line $\{q_0\}
\times [-\frac{\pi}{2}, \frac{\pi}{2}]$. To construct a horseshoe
respecting $H_\sigma$, it is necessary to view two nearby points as
having {\it separated} when they lie on opposite sides of $\partial
H_\sigma$ or on opposite sides of $H_\sigma$ in $M \setminus
H_\sigma$. Thus it is convenient to view $f^{-1}(\partial H_\sigma)$
as part of the discontinuity set of $f$. For simplicity, consider
first the case where $q_0$ does not lie on a line in the table X
tangent to more than one scatterer. Then $f^{-1} (\partial
H_\sigma)$ is a finite union of pairs of roughly parallel, smooth
$s$-curves. (Recall that $s$-curves are negatively sloped, with
slopes uniformly bounded away from $0$ and $- \infty$.) Each of the
curves comprising $f^{-1} (\partial H_\sigma)$ begins and ends in
$\partial M \cup f^{-1}(\partial M)$, that is to say, the geometric
properties of $f^{-1}(\partial H_\sigma) \cup f^{-1} (\partial  M)$
are similar to those of $f^{-1} (\partial M)$. Likewise, $f(\partial
H_\sigma)$ is a finite of union of pairs of (increasing) $u$-curves
that begin and end in $\partial M \cup f(\partial M)$, and it will
be convenient to regard that as part of the discontinuity set of
$f^{-1}$.

Let $N_{\varepsilon}(\cdot)$ denote the $\varepsilon$-neighborhood
of a set.  We will need the following lemma.

\begin{lemma}
For each $\varepsilon > 0$ there is an $h > 0$ such that for each
$\sigma \in \Sigma_h$, $H_{\sigma} \subset N_{\varepsilon}(H_0)$,
$fH_{\sigma} \subset N_{\varepsilon}(fH_0)$, and $f^{-1}H_{\sigma}
\subset N_{\varepsilon}(f^{-1}H_0)$. \end{lemma}

As $f$ is discontinuous, Lemma 3.1 is not immediate.  However, it can
be easily verified, and we leave the proof to the reader.

Points $q_0$ that lie on lines in $X$ with multiple tangencies to
scatterers lead to slightly more complicated geometries, and special
care is needed when defining what is meant by $fH_0$ and
$f^{-1}H_0$.  For example, consider the case where $q_0 \in
\Gamma_3$ lies on a line that is tangent to $\Gamma_1$ and
$\Gamma_2$, but which is not tangent to any other scatterer
including $\Gamma_3$.  Suppose further that $r_1 \in \Gamma_1$, $r_2
\in \Gamma_2$ are the points of tangency, that $r_2$ is closer to
$q_0$ than $r_1$ is, that no other scatterer touches the line
segment $[q_0, r_1]$, and that $\Gamma_1$ and $\Gamma_2$ both lie on
the same side of $[q_0, r_1]$; see Fig.~2~(left). Let $\sigma$ be a
small hole of Type I with $q_0 \in \sigma$.  Then in $\Gamma_2
\times [-\pi/2, \pi/2]$, $f^{-1}(\partial H_{\sigma})$ appears as
described above.  However, $\Gamma_2$ ``obstructs'' the view of
$\sigma$ from $\Gamma_1$, and so in $\Gamma_1 \times [-\pi/2,
\pi/2]$, $f^{-1}(H_{\sigma})$ is a small triangular region whose
three sides are composed of a segment from $\Gamma_1 \times
\{\pi/2\}$, a segment from $f^{-1}(\Gamma_2 \times \{\pi/2\})$, and
a \emph{single} segment from $f^{-1}(\partial H_{\sigma})$. See
Fig.~2~(right). As a consequence, when we write $f^{-1} H_0$, we
include in this set not just $(r_2, \pi/2)$, but also $f^{-1}(r_2,
\pi/2) = (r_1, \pi/2)$. This is necessary in order for Lemma 3.1 to
continue to hold.  Aside from such minor modifications, the case of
multiple tangencies is no different than when they are not present,
and we leave further details to the reader.

\begin{figure}[htp]
\centering
\resizebox{1.5 in}{!}{
\includegraphics{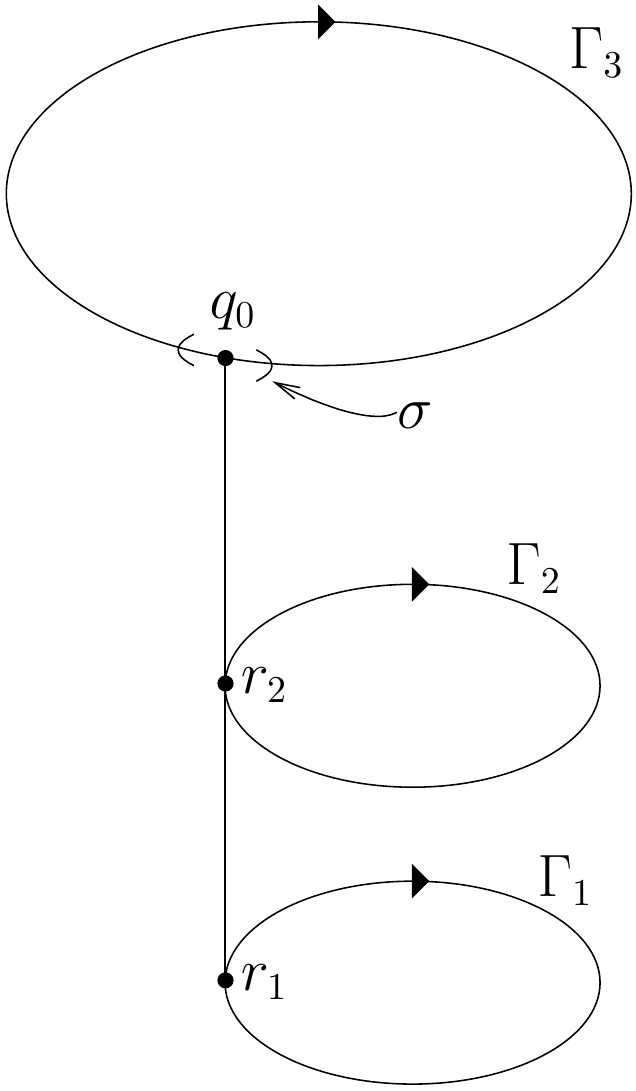}}
\hspace{.35 in}
\resizebox{4.5 in}{!}{
\includegraphics{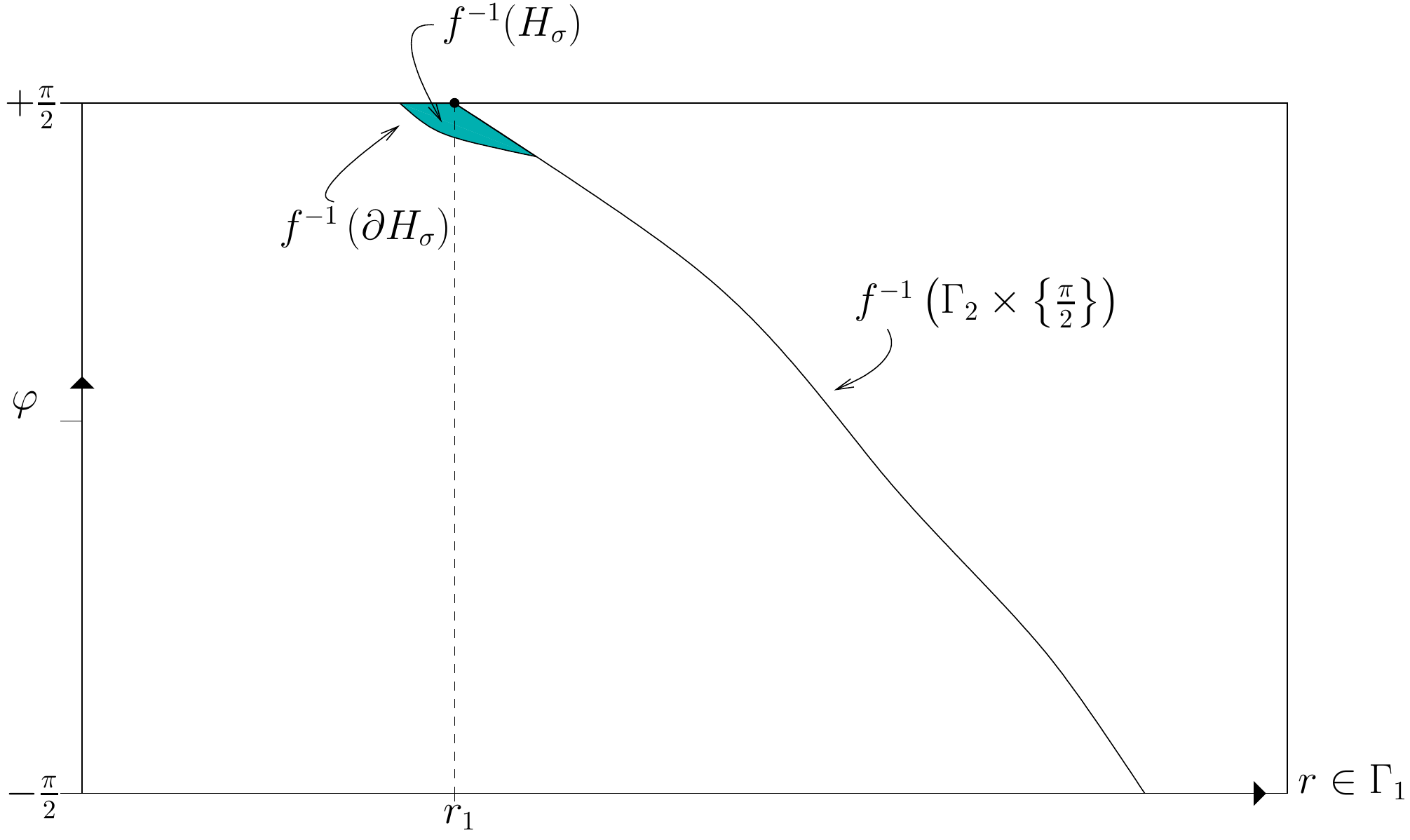}}
\label{figure 2} \caption{An infinitesmal hole aligned with multiple
tangencies. {\it Left:} $q_0$ lies on a line segment in the
billiard table $X$ that is tangent to two scatterers. {\it Right:}
Induced singularity curves in the subset $\Gamma_1 \times [-\pi/2,
\pi/2]$ of the phase space $M$.}
\end{figure}

\medskip
\noindent {\bf Holes of Type II.} For simplicity, consider first the
case where $q_0$ does not lie on a line in the ``table" $X$ tangent
to more than one scatterer. Recall from Sect.~1.2 that ``the hole''
$H_\sigma$ here is taken to be $f(B_\sigma)$ where $B_\sigma$
consists of points in $M$ which enter $\sigma \times \mathbb{S}^1$
under the billiard flow before returning to the section $M$. As with
holes of Type I, we define $\partial H_\sigma$ to be the boundary of
$H_\sigma$ viewed as a subset of $M$. The set $B_\sigma$ as a subset
of $M$ has similar geometric properties as $f^{-1}H_\sigma$ for Type
I holes, {\it i.e.}, $f^{-1}(\partial H_\sigma) \backslash (\partial M \cup f^{-1}( \partial M))$
consists of pairs
of negatively sloped curves ending in $\partial  M \cup
f^{-1}(\partial M)$. The slopes of these curves are uniformly
bounded (independent of $\sigma $) away from $- \infty$ and $0$.
For the reasons discussed, it will be convenient to view this set as
part of the discontinuity set of $f$. The infinitesimal hole $H_0
\subset M$ is defined in the natural way, and the analog of Lemma
3.1 can be verified.  We will say more about the geometry of
$H_\sigma$ in Sect.~3.3.

Points $q_0$ that lie on multiple tangencies lead to slightly more
complicated geometries, and special care is needed when defining
what is meant by the sets $f^{-1}H_0$, $H_0$, and $f H_0$ as in the
case of Type I holes.

\medskip
\noindent {\it Further generalizations on holes of Type II:} In
addition to the generalizations discussed in Sect.~1.3, sufficient
conditions on the holes allowed in $\Sigma_h$ for Prop. 2.2 to
remain true are the following, as can be seen from our proofs:

\smallskip
\noindent (1) There exist $N$ and $L$ for which the following hold
for all sufficiently small $h$:

\begin{enumerate} \vspace{-10 pt}
    \item[(a)]  $f^{-1} (\partial H_\sigma)$, $\partial H_\sigma$, and
$f(\partial H_\sigma)$ each consist of no more than $N$ smooth
curves, all of which have length no greater than $L$. \vspace{-8 pt}
    \item[(b)] For each $\sigma \in \Sigma_h$,
    $f^{-1} (\partial H_\sigma) \backslash (\partial M \cup f^{-1}( \partial M))$
consists of piecewise smooth, negatively sloped curves
(with slopes uniformly bounded away from $- \infty$ and $0$),
and the end points  of these curves must lie on $\partial M \cup f^{-1}(\partial M)$.
\end{enumerate}

\vspace{-8 pt}
\noindent
(2) The analog of Lemma 3.1 holds.

\smallskip

\noindent Thus it would be permissible to allow a convex hole
$\sigma$ to be in $\Sigma_h$ that did not have a ${C}^3$ simple
closed curve with strictly positive curvature as its boundary. For
example, conditions (a) and (b) above hold if $\partial \sigma$ is a
piecewise $C^3$ simple closed curve which consists of finitely many
smooth segments that are either strictly positively curved or flat.
As another generalization, consider the case when any line segment
in the table $X$ with its endpoints on two scatterers that passes
through the convex hull  of $\sigma$  also intersects $\sigma$. Then
it is no loss of generality to replace $\sigma$ by its convex hull.
Using this, one can often verify that the set $H_\sigma$ that arises
satisfies properties (a) and (b) above, even if $\sigma$ is not
itself convex. See Fig.~3.

\begin{figure}[htp]
\centering
\resizebox{5 in}{!}{
\includegraphics{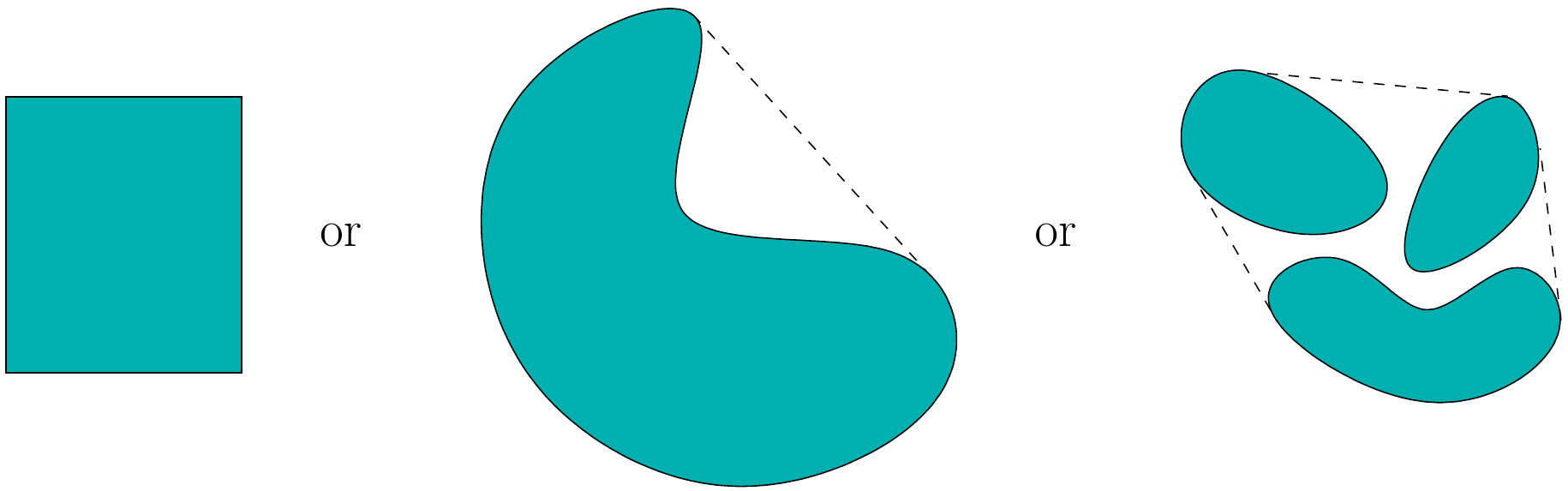}}
\label{figure 3}
\caption{Examples of Type II holes that are permissible.}
\end{figure}

\medskip
In Sect. 3.2, the discussion is for holes of Type I with a single
interval deleted. The proof follows {\it mutatis mutandis} for holes
of Type II, with the necessary minor modifications discussed in
Sect.~3.3.

%%%%%%%%%%%%%%%%%%%%%%%%

\subsection{Proof of Proposition~\ref{prop:horseshoe} (for
holes of Type I)} \label{proof of horseshoe}

The idea of the proof is as follows. First we construct a horseshoe
$(\Lambda^{(0)}, R^{(0)})$ with the desired properties for the {\it
infinitesimal hole} $\{q_0\}$. Then we construct
$(\Lambda^{(\sigma)}, R^{(\sigma)})$ for all $\sigma \in
\Sigma_h(q_0)$, and show that with $\Lambda^{(\sigma)}$ sufficiently
close to $\Lambda^{(0)}$ in a sense to be made precise,
$(\Lambda^{(\sigma)}, R^{(\sigma)})$ will inherit the desired
properties with essentially the same bounds. To ensure that
$\Lambda^{(\sigma)}$ can be taken ``close enough" to
$\Lambda^{(0)}$, we decrease the size of the hole, {\it i.e.}, we
let $h \to 0$. Now the constructions of
 $(\Lambda^{(0)}, R^{(0)})$ and
$(\Lambda^{(\sigma)}, R^{(\sigma)})$ are essentially identical. To
avoid repeating ourselves more than needed, we will carry out the
two constructions simultaneously. It is useful to keep in mind,
however, that logically, the case of the infinitesimal hole is
treated first, and some of the information so obtained is used to
guide the arguments for positive-size holes.

As explained in Sect.~3.1, to ensure that the horseshoe respects the
hole, it is convenient to include $f^{-1}(\partial H_\sigma)$ as
part of the discontinuity set for $f$. Since $H_\sigma$ will be
viewed as a perturbation of $H_0$, we include $f^{-1}(H_0)$ in this
set as well. The following convention will be adopted when we
consider a system with hole $H_\sigma$:
\begin{itemize}
\item[(a)] Suppose for definiteness $q_0 \in \Gamma_1$. The new phase
space $M_\sigma$ is obtained from $M$ by cutting $\Gamma_1 \times
[-\frac{\pi}{2}, \frac{\pi}{2}]$ along the lines comprising
 $H_0 \cup \partial
H_\sigma$, splitting it into three connected components.
\item[(b)] As a consequence, the new discontinuity set of $f$ is
$f^{-1}(\partial M_\sigma)$, and the new discontinuity set of
$f^{-1}$ is $f(\partial M_\sigma)$.
\end{itemize}
We use the notation ``$\sigma = 0$" for the infinitesimal hole, so
that $M_0$ is obtained from $M$ by cutting along $H_0$.

Notice immediately that this changes the definitions of stable and
unstable curves, in the sense that if $\gamma$ was a stable curve
for the system without holes, then $\gamma$ continues to be a stable
curve if and only if (i) $\gamma \cap \partial M_\sigma =
\emptyset$, and (ii)  $f^n(\gamma) \cap f^{-1}\partial M_\sigma =
\emptyset$ for all $n \ge 0$; a similar characterization holds for
unstable curves. All objects constructed below will be
$\sigma$-dependent, but we will suppress mention of $\sigma$ except
where it is necessary. Observe also that the Important Geometric Facts
($\dagger$) in Sect.~2.3 with $f^{-1}\partial M_\sigma$ instead of
$f^{-1}\partial M$ as the new discontinuity set remains valid.

We now follow sequentially the 7 points outlined in Sect.~2.3 and
discuss the modifications needed. These modifications, along with
two additional points (8 and 9) form a complete proof of Proposition
2.2. We believe we have prepared ourselves adequately in Sects.~2.3
and 2.4 so that the discussion to follow can be understood on its
own, but encourage readers who wish to see proofs complete with all
technical detail to read the rest of this section alongside the
relevant parts of [Y] and [BSC2].

The notation within each item below is as in Sect.~2.3.

\bigskip
\noindent 1. The relationships $\lambda=\lambda_1^4$ and
$\delta=\delta_1^4$ are as before, and the sets $B^{(\sigma)
\pm}_{\lambda_1, \delta_1}$ are defined in a manner similar to that
in Sect.~2.3. For example,
$$
B^{(\sigma) +}_{\lambda_1, \delta_1} = \{x \in M_\sigma: d(x,
\partial M_\sigma) \ge \delta_1 \ {\rm and} \ d(f^nx,
f^{-1}\partial M_\sigma) \ge \delta_1 \lambda_1^{-n} \ \forall
n \ge 0\}.
$$
As in Sect.~2.3, the condition on $d(f^nx, f^{-1}\partial M_\sigma)$
is to ensure the existence of stable curves, and the
necessity for $x$ to be away from $\partial M_\sigma$ is obvious
(cf. footnote in item 1 of Section 2.3).
Properties (i) and (ii) continue to hold for each $\sigma$ given the
geometry of the new discontinuity set. With regard to the choice of
$\delta_1$, we let $\delta_1$ be as in [Y], and shrink it if
necessary to ensure that $B^{(0) +}_{\lambda_1, 2\delta_1} \cap
B^{(0) -}_{\lambda_1, 2\delta_1}$ has positive $\nu$-measure away
from $f^{-1}(\partial M_0) \cup \partial M_0 \cup f(\partial M_0)$.
This is where the sets $\Lambda^{(\sigma)}$ will be located (see
Paragraph 2).

The following lemma relates $B_{\lambda_1, \delta_1}^{(\sigma) \pm}$
and $B_{\lambda_1, \delta_1}^{(0) \pm}$:

\begin{lemma} (i) For all $\sigma \in \Sigma_h$, we have
$B_{\lambda_1, \delta_1}^{(\sigma) \pm} \subset B_{\lambda_1,
\delta_1}^{(0) \pm}$.

\noindent (ii) As $h \to 0$,
$$
\sup_{\sigma \in \Sigma_h} \nu(B_{\lambda_1, \delta_1}^{(0) +}
\setminus B_{\lambda_1, \delta_1}^{(\sigma) +})\ \to \ 0, \qquad
\sup_{\sigma \in \Sigma_h} \nu(B_{\lambda_1, \delta_1}^{(0) -}
\setminus B_{\lambda_1, \delta_1}^{(\sigma) -})
 \ \to \ 0.
$$
\end{lemma}

\noindent {\bf Proof:} (i) follows immediately from $\partial
M_\sigma \supset \partial M_0$. As for (ii), let $\varepsilon > 0$
be given. Recall that $N_\alpha(\cdot)$ denotes the
$\alpha$-neighborhood of a set. By Lemma 3.1 we may choose $h$ small
enough that for all $\sigma \in \Sigma_h$, $\partial H_\sigma \in
N_\varepsilon(H_0)$ and $f^{-1}(\partial H_\sigma) \in
N_\varepsilon(f^{-1}H_0)$. Then if $x\in
B_{\lambda_1,\delta_1}^{(0)+} \setminus
B_{\lambda_1,\delta_1}^{(\sigma)+}$, either $x\in
N_{\delta_1}(\partial H_\sigma) \setminus N_{\delta_1}(H_0)$, or
$$x\in \cup_{n\ge0}f^{-n}(N_{\delta_1\lambda_1^{-n}}(f^{-1}\partial H_\sigma)
\setminus N_{\delta_1\lambda_1^{-n}}(f^{-1}H_0)).
$$
We estimate the $\nu$-measure of the right side separately for
$\cup_{n \ge n_\varepsilon}$ and $\cup_{n< n_\varepsilon}$ where
$n_{\varepsilon} = \inf \{n\geq 0 : \delta_1 \lambda_1^{-n} \leq
\eps \} \approx \frac{\ln(\frac{\delta_1}{\eps})}{\ln \lambda_1}$.
For $n \ge n_\varepsilon$, the measure in question is
$$
\le \sum_{n \geq n_{\eps}} \nu(N_{\delta_1
\lambda_1^{-n}}(f^{-1}(\partial H_{\sigma})) \le {\rm const} \cdot
\sum_{n \geq n_{\eps}} \delta_1\lambda_1^{-n} \le {\rm const} \cdot
\varepsilon\ .
$$
Here we have used that $f^{-1}(\partial H_{\sigma})$ consists of a
finite number of smooth compact curves the total length of which is
bounded independent of $\sigma$.  Adding to this that these curves
are within a distance $\varepsilon$ of the curves in $f^{-1}H_0$, we
see that for each $n< n_{\varepsilon}$,
$$
\nu(N_{\delta_1 \lambda_1^{-n}}(f^{-1}(\partial H_{\sigma}))
\setminus N_{\delta_1 \lambda_1^{-n}} (f^{-1}H_0)) \le {\rm const}
\cdot \varepsilon.
$$
Similarly, $\nu(N_{\delta_1}(\partial H_\sigma) \setminus
N_{\delta_1}(H_0)) \le {\rm const} \cdot \varepsilon$. Hence
$$\nu(B_{\lambda_1,\delta_1}^{(0) +} \setminus
B_{\lambda_1, \delta_1}^{(\sigma) +}) \le {\rm const} \cdot
(\varepsilon + (n_{\varepsilon}+1) \varepsilon) \le {\rm const}
\cdot \varepsilon \ln(\frac{\delta_1}{\varepsilon})
$$
which tends to $0$ as $\varepsilon \to 0$. \hfill $\square$

\bigskip

\noindent 2. To construct the Cantor sets, we first pick $x_1^{(0)}$
as a density point of $B^{(0) +}_{\lambda_1, 2\delta_1} \cap B^{(0)
-}_{\lambda_1, 2\delta_1}$ at least $2 \delta_1$ away from
$f^{-1}(\partial M_0) \cup \partial M_0 \cup f(\partial M_0)$ and
the boundaries of the homogeneity strips, and begin to construct
$\Lambda^{(0)}$ with $\Omega = W^u_\delta(x_1^{(0)})$. We then do
the same for each $\sigma$, {\it i.e.}, pick $x_1^{(\sigma)}$ as a
density point of $B^{(\sigma) +}_{\lambda_1, 2\delta_1} \cap
B^{(\sigma) -}_{\lambda_1, 2\delta_1}$ and begin to construct
$\Lambda^{(\sigma)}$ centered at $x_1^{(\sigma)}$ -- except that for
reasons to become clear, we will want $d(x_1^{(\sigma)}, x_1^{(0)})<
\delta_2$ where $\delta_2>0$ is determined by properties of
$(\Lambda^{(0)}, R^{(0)})$ (requirements will appear below, and in
items 6 and 9). Suffice it to say here that however small $\delta_2$
may be, Lemma 3.2 guarantees that this can be done by shrinking $h$.
Once $x_1^{(\sigma)}$ is chosen, we set $\Omega =
W_\delta^u(x_1^{(\sigma)})$ and
$$
\Omega_n = \{ y \in \Omega : d(f^iy, f^{-1}(\partial M_\sigma)) \ge
\delta_1\lambda_1^{-i} \ {\rm for }\ 0\leq i\leq n\}.
$$
Then the sets $\Omega_\infty$, $\Gamma^s$, $\Gamma^u$ and $\Lambda^{(\sigma)}$ are constructed as before.

That $Q(\Lambda^{(\sigma)}) \approx Q(\Lambda^{(0)})$ follows
immediately from the proximity of $x_1^{(\sigma)}$ to $x_1^{(0)}$.
Since $\delta$ is fixed, $\mu^u(\Lambda ^{(\sigma)}) \approx
\mu^u(\Lambda ^{(0)})>0$ can be arranged by taking $\delta_2$
sufficiently small and using Lemma 3.2 with $h$ sufficiently small.
This proves Proposition 2.2(b)(i). With the separation time
happening sooner due to the enlarged discontinuity set, {\bf
(P3)}--{\bf (P5)} remain true with the same $C$ and $\alpha$ for the
closed system; in other words, Proposition 2.2(b)(iii) requires no
further work.

\bigskip
\noindent 3.  To arrange for mixing properties (not done in [Y]), we
will need to delay the return times to $\Lambda$ by forbidding
returns before time $R_2$ for some $R_2 \ge R_1$ determined by
$(\Lambda^{(0)}, R^{(0)})$; see Lemma 3.4. This aside, the
construction of $f^R$ is as before. The matching of Cantor sets
argument should be looked at again since the Cantor sets are
different, but the proof goes through as before because the sets are
dynamically defined.

Notice that for $\omega \in \tilde {\cal P}_n$, $f^i\omega$ is
either entirely in the hole or outside of the hole, as is
$f^i(\Lambda^s)$ where $\Lambda^s$ is the $s$-subset of $\Lambda$
associated with $\omega$, for $0 \leq i \leq n$; this is a direct
consequence of our taking the boundary of the hole into
consideration in our definition of the discontinuity set. Together
with the fact that $\Lambda$ is away from $\partial M_\sigma$, it
ensures that the generalized horseshoe we are constructing respects
the hole.

\bigskip
\noindent 4. This is where one of the more substantial modifications
occur: Lemma 2.3, which is based largely on the competition between
expansion along $u$-curves and the rate at which they are cut, is
clearly affected by the additional cutting due to our enlarged
discontinuity set. The condition (*) in Sect.~2.3 must now be
replaced by

\begin{lemma} There exists $K_1$ such that for any
$m \in {\mathbb Z}^+$, there exists $\eps_0>0$ with the property
that for any $u$-curve with $p(\omega)< \eps_0$, $f^m(\omega)$ has
$\le (K_1m^2+4)$ connected components with respect to the enlarged
discontinuity set.
\end{lemma}

\noindent {\bf Proof:} Let $m \in {\mathbb Z}^+$ be given. As in
Sect.~2.3, choose $\eps_0 > 0$ small enough such that if $\omega$ is
a $u$-curve with $p (\omega) < \eps_0$, $f^m(\omega)$ has $\le (K_0
m + 1)$ connected components with respect to the original
discontinuity set $f^{-1}S_0$. Let $\omega_j$ be the $f^{-m}$-image
of one these connected components. This means that for $ 0 \le k\le
m$, $f^k(\omega_j)$ is, in reality, a connected $u$-curve even
though it may not be connected with respect to our enlarged
discontinuity set. Since $f^k \omega_j$ is an (increasing)
$u$-curve, it can meet the three vertical lines making up $(\partial
H_{\sigma}) \cup H_0$ in no more than three points.
(As the slopes $d \varphi / dr$ of $u$-curves are never less than the
curvature of $\Gamma_i$ at $r$, connected $u$-curves cannot wrap around
the cylinder $\Gamma_i \times [ -\pi / 2 , \pi / 2 ]$ and meet
$(\partial H_\sigma) \cup H_0$ more than once.)
Hence the
cardinality of $\{\omega_j \cap \bigcup_{k=0}^m f^{-k}((\partial
H_{\sigma}) \cup H_0 ) \} $ is $\leq 3(m+1)$, and as $(\partial
H_{\sigma}) \cup H_0$ is the additional set added to $\partial M$ to
create $\partial M_\sigma$, it follows that $f^m \omega$ has $\le
(K_0 m + 1) \cdot (3(m+1) + 1)$ connected components with respect to
the enlarged discontinuity set. \hfill $\square$

\medskip
Using Lemma 3.3, one adapts easily the proof of Lemma 2.3 to the
present setup with $\theta_1 = (K_1m^2+4)^{\frac{1}{m}}
(\lambda^{-1} + \alpha_0)$, where $m$ is chosen large enough so that
this number is $<1$. The constant $D_1$ depends only on the
properties of $Df$ and is unchanged. Hence Lemma 2.3 is valid with
$D_1$ and $\theta_1$ modified but independent of $\sigma$.
As in Section 2.3, these estimates can then be adapted to estimate
$p( \omega_n \backslash \{ T \le n \})$.

\bigskip
\noindent 5. Lemma 2.4 remains valid with modified constants which
are independent of $\sigma$.  Returning to the sketch of the proof
provided in Sect. 2.3, we see that both sets of estimates boil down
to  the geometry of the new discontinuity set and the rates of
growth {\it versus} cutting, which has been taken care of for the
enlarged discontinuity set in Paragraph 4 above.

\bigskip
\noindent 6. We need to show that there exist $n_1$ and
$\varepsilon_1>0$ independent of $\sigma$ such that for every
homogeneous $u$-curve with $p$-length $> \eps_0$, a fraction $\ge
\varepsilon_1$ of $\omega$ returns within the next $n_1$ steps.
Before we enlarged the discontinuity set, this property followed
from property (**) in Sect.~2.3.  We replace (**) here with the
following:

\begin{lemma} Given $\varepsilon_0 > 0$, provided $h$ and $\delta_2$
are sufficiently small, there exists $n_1$ such that the following
holds for each $\sigma \in \Sigma_h$: for every homogeneous
$W_{\text{loc}}^u$-curve $\omega$ with $p(\omega) > \varepsilon_0$
and each $q \in \{n_1, n_1 + 1 \}, f^q \omega$ contains a
homogeneous segment that u-crosses the middle half of
$Q(\Lambda^{(\sigma)})$ with greater than $2\delta$ sticking out
from each side.
\end{lemma}

Once Lemma 3.4 is proved, the fact that a fraction $\varepsilon_1$
(independent of $\sigma$) has the desired properties follows from
derivative estimates as in Sect.~2.3 and our uniform lower bound on
$\mu^u(\Lambda^{(\sigma)})$.  The reason we want $q$ to take two
consecutive values in the statement of Lemma 3.4 has to do with the
mixing property in item 9 below.

\medskip
\noindent {\bf Proof:} Fix $\varepsilon_0 > 0$.  We first prove the
following for the case $\sigma=0$:

\begin{itemize}
\item[(**)']
 {\it  For $\sigma = 0$, there exists $n_1$ such that
any homogeneous $W_{\text{loc}}^u$-curve $\omega$ with $p(\omega) >
\varepsilon_0$ and every $q \geq n_1$, $f^q\omega$ contains a
homogeneous segment that u-crosses the middle fourth of
$Q(\Lambda^{(0)})$ with greater than $4 \delta$ sticking out from
each side.  }
\end{itemize}
The proof of (**)' is completely analogous to the proof of (**)
outlined in Sect.~2.4. Sublemmas A and B continue to hold due to the
similar geometry of the discontinuity set.

Notice that unlike (**)', the assertion in Lemma 3.4 is only for
$q=n_1$ and $n_1+1$, so that its proof involves only a finite number
of mixing boxes $U_j$ and a finite number of iterates. This will be
important in the perturbative argument to follow.

Consider now $\sigma \neq 0$, and consider a homogeneous unstable
curve $\omega$ with $p(\omega) > \varepsilon_0$. First, $\omega$
continues to be an unstable curve with respect to the discontinuity
set $f^{-1}\partial M_0$, so by the proof of (**)', for $q \in \{n_1, n_1 + 1\}$
and every $j$, there is a rectangular region $Q^*=Q^*(q,j)$ such
that (i) $Q^*$ $u$-crosses the middle fourth of $Q(\Lambda^{(0)})$
with $> 4\delta$ sticking out, (ii) $f^{-q}Q^*$ is an
$s$-subrectangle in the middle third of $Q(U_j)$, and (iii) for
$i=0,1, \cdots, q$, $f^{-i}Q^*$ stays clear of $f^{-1}\partial M_0$
by some amount. Lemma 3.1 ensures that for $h$ small enough, (iii)
continues to hold with $f^{-1}\partial M_0$ replaced by
$f^{-1}\partial M_\sigma$. Finally, provided $\delta_2$ is small
enough, (i) holds for $Q(\Lambda^{(\sigma)})$ with $> 2\delta$ sticking out on each side.
\hfill $\square$

\bigskip
\noindent 7. Once steps 4, 5 and 6 have been completed, the argument
here is unchanged (as it is largely combinatorial), guaranteeing
constants $C_0$ and $\theta_0$ independent of $\sigma$ with $p\{R
\ge n\} \le C_0\theta_0^n$. This completes the proof of Proposition
2.2(a)(i) and (b)(ii).

\bigskip
We have reached the end of the 7 steps outlined in Sect.~2.3. Two
items remain:

\medskip
\noindent 8. That $\bar n(h) \to \infty$ as $h \to 0$ is easy:
Orbits from $\Lambda^{(\sigma)}$ start away from $H_0$ and cannot
approach $f^{-1}H_0$ faster than a fixed rate. Thus using Lemma 3.1, we
can arrange for orbits starting from $\Lambda^{(\sigma)}$ to stay
out of $H_\sigma$ for as long as we wish by taking $h$ small.

\bigskip
\noindent 9. The mixing of $(\Lambda^{(\sigma)}, R^{(\sigma)})$
follows from

\begin{lemma} There exists $R_2 \ge R_1$ (independent of
$\sigma$) such that for small enough $h$, the construction in Step 3
can be modified to give the following:

(i) no returns are allowed before time $R_2$, and

(ii) at both times $R_2$ and $R_2+1$, there are $s$-subsets of
$\Lambda^{(\sigma)}$ making full returns.
\end{lemma}

\noindent {\bf Proof:} Again we first consider the case $\sigma =
0$. Here $R_2$ is chosen as follows: Without allowing any returns,
let $R_1'$ be the smallest time greater than or equal to $R_1$ such
that there exists $\omega \in \tilde{\cal P}_{R_1'}$ with
$p(f^{R_1'} \omega) > \eps_0>0$. With $\varepsilon_0$ chosen as
before, we take $n_1$ from Lemma 3.4 and set $R_2 = R_1' + n_1$.
Using Lemma 3.4, we find two subsegments $\omega'$ and $\omega''
\subset \omega$ such that $f^{R_2}\omega'$ and $f^{R_2+1}\omega''$
are both homogeneous segments that $u$-cross the middle half of
$Q(\Lambda^{(0)})$ with greater than $2\delta$ sticking out from
each side.  We may suppose that $\omega'$ and $\omega''$ are
disjoint since $f$ has no fixed points.  They give rise to two
$s$-subsets of $\Lambda^{(0)}$ with the properties in (ii).  From
time $R_2$ on, returns to $\Lambda^{(0)}$ are allowed as before.

When $\sigma \not= 0$, we follow the same procedure as above to
ensure the mixing of $(\Lambda^{(\sigma)}, R^{(\sigma)})$.  The only
concern is that $R_1'=R_1' (\sigma)$ (and hence also $R_2 = R_1' +
n_1$) might not be independent of $\sigma$.  This is not a problem
as the construction above involves only a finite number of steps:
With $h$ and $\delta_2$ sufficiently small, the elements of
$\tilde{\cal P}_n^{(\sigma)}$ can be defined in such a way that they
are in a one-to-one correspondence with those of $\tilde{\cal
P}_n^{(0)}$ for $n\leq R_1'(0)$. \hfill $\square$

\medskip

Finally, mixing of the surviving dynamics is ensured by choosing $h$
small enough that $\bar n(h)>R_2+1$. This ensures that the
$s$-subsets $\Lambda^s$ that make full returns at times $R_2$ and
$R_2+1$ cannot fall into the hole prior to returning.

\bigskip
The proof of Proposition 2.2 for holes of Type I is now complete.

%%%%%%%%%%%%%%%%%%%%%%%%

\subsection{Modifications needed for holes of Type II}

The proof for Type II holes is very similar to that for Type I
holes. There are, however, some differences due to the more
complicated geometry of $\partial H_\sigma$. In the discussion
below, we assume the infinitesimal hole $\{q_0\}$ does not lie on
any segment in the table tangent to more than one scatterer. The
general situation is left to the reader.

From the discussion of the geometry of Type II holes in Sect.~3.1,
we see that the Important Geometric Facts ($\dagger$) in Sect.~2.3
continue to hold with $M_\sigma$ in the place of $M$, except that
$u$-curves need not be transversal to the $\partial H_\sigma \cup
H_0$ part of $\partial M_\sigma$. Potential problems that may arise
are discussed below. The discontinuity set of $f$, i.e. $f^{-1}
\partial M_\sigma$, has the same geometric properties as before.

We now go through the 9 points in Sect.~3.2. No modifications are
needed in items 1--3. As expected, item 4 is where the most
substantial modifications occur:

\medskip
\noindent {\bf Modifications in Item 4.} Lemma 3.3 is still true as
stated, but the geometry is different. In the discussion below related to
this lemma, the discontinuity set refers to $f^{-1}\partial M$, not the
enlarged discontinuity set $f^{-1}\partial M_\sigma$, and unstable
curves are defined accordingly. For Type I holes, the proof relies
on the fact that any (increasing) connected $u$-curve $\omega$ meets
$\partial H_\sigma \cup H_0$, which is the union of three vertical
lines, in at most three points.

\begin{lemma} Any unstable curve $\omega$ meets
$\partial H_\sigma \cup H_0$ in at most three points.
\end{lemma}

Even though Lemma 3.3 is stated for $u$-curves, we need it only for
unstable curves (and the argument here is slightly simpler for
unstable curves).

\begin{proof}
Let us distinguish between two different types of curves that
comprise $\partial H_\sigma$: {\it Primary} segments, which are the
forward images of curves in $\partial B_\sigma \setminus
f^{-1}(\partial M)$, and {\it secondary} segments, which are
subsegments of $f(\partial M)$.  For examples, see Fig.~4.  In
general, when $q_0$ does not lie on a line segment with multiple
tangencies to the scatterers, secondary segments are absent in
$H_0$, while each component of $H_0$ gives rise to two primary
segments in $\partial H_\sigma$ for $\sigma \ne 0$.

\begin{figure}[htp]
\centering
\hspace{.2 in}
\resizebox{!}{2.7 in}{
\includegraphics{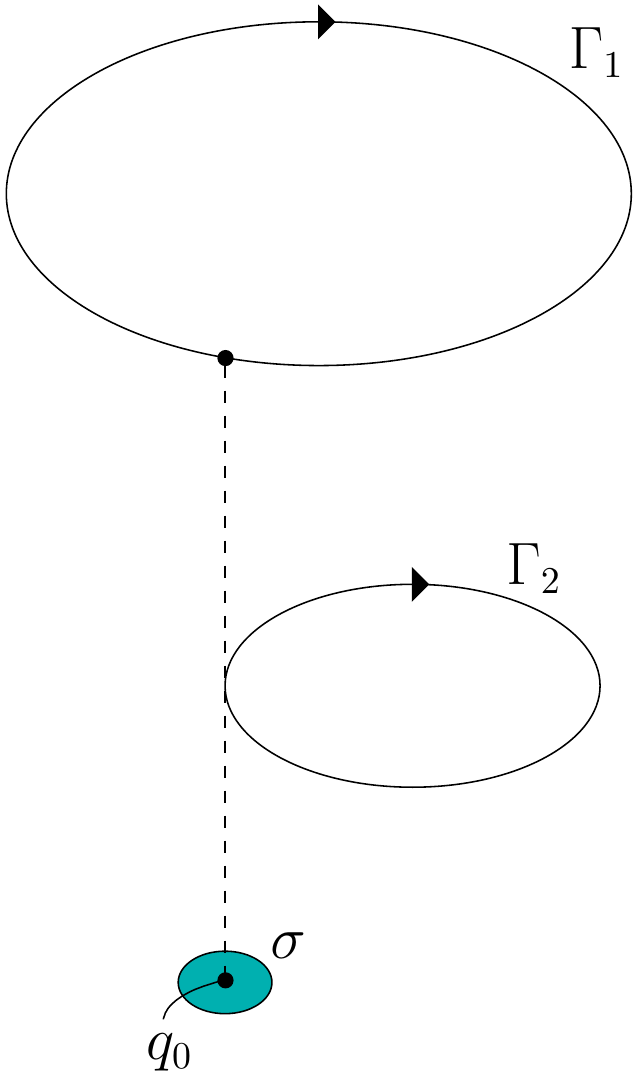}}
\hspace{.45 in}
\resizebox{!}{2.7 in}{
\includegraphics{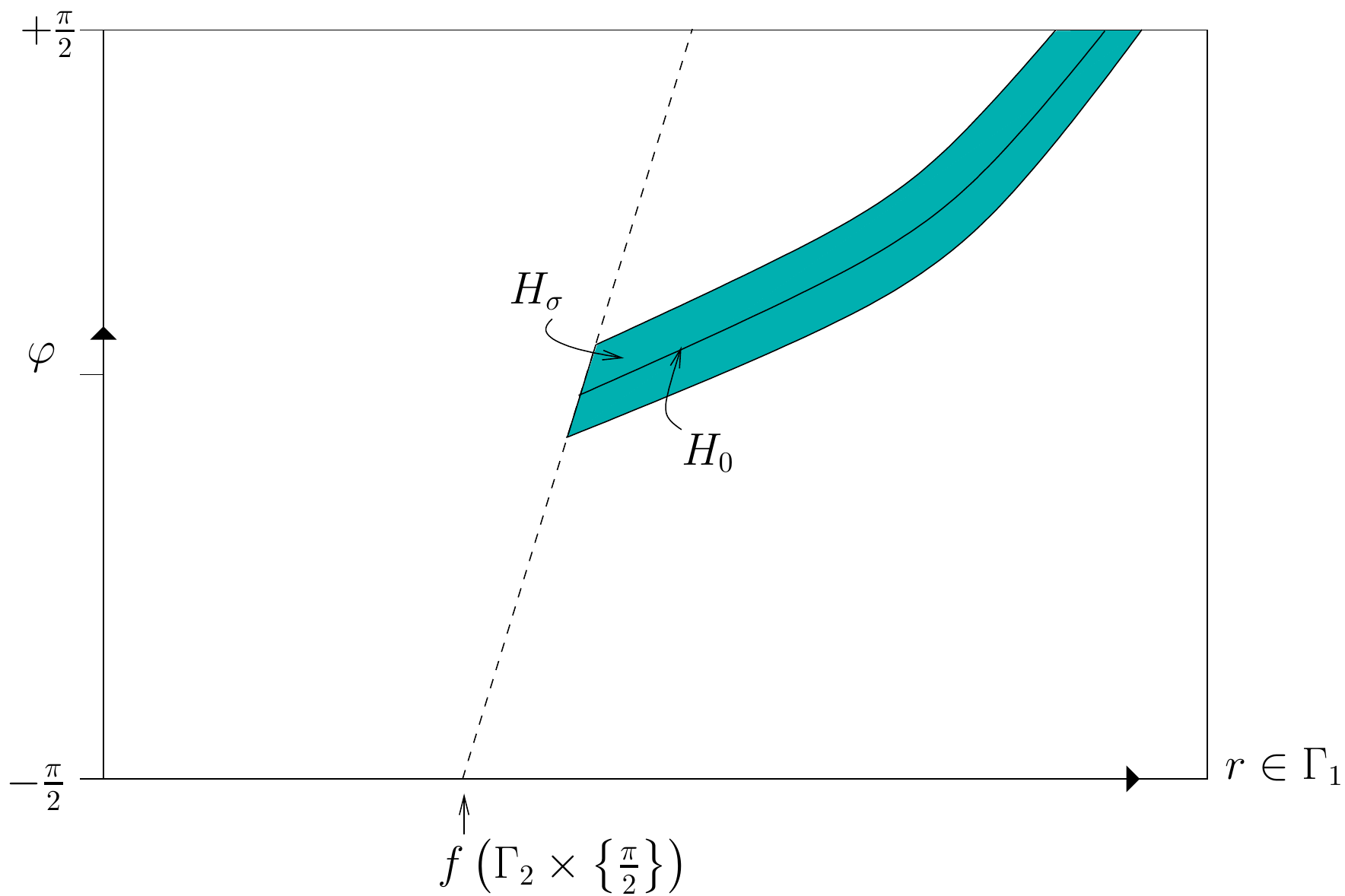}}

\bigskip

\resizebox{!}{2.7 in}{
\includegraphics{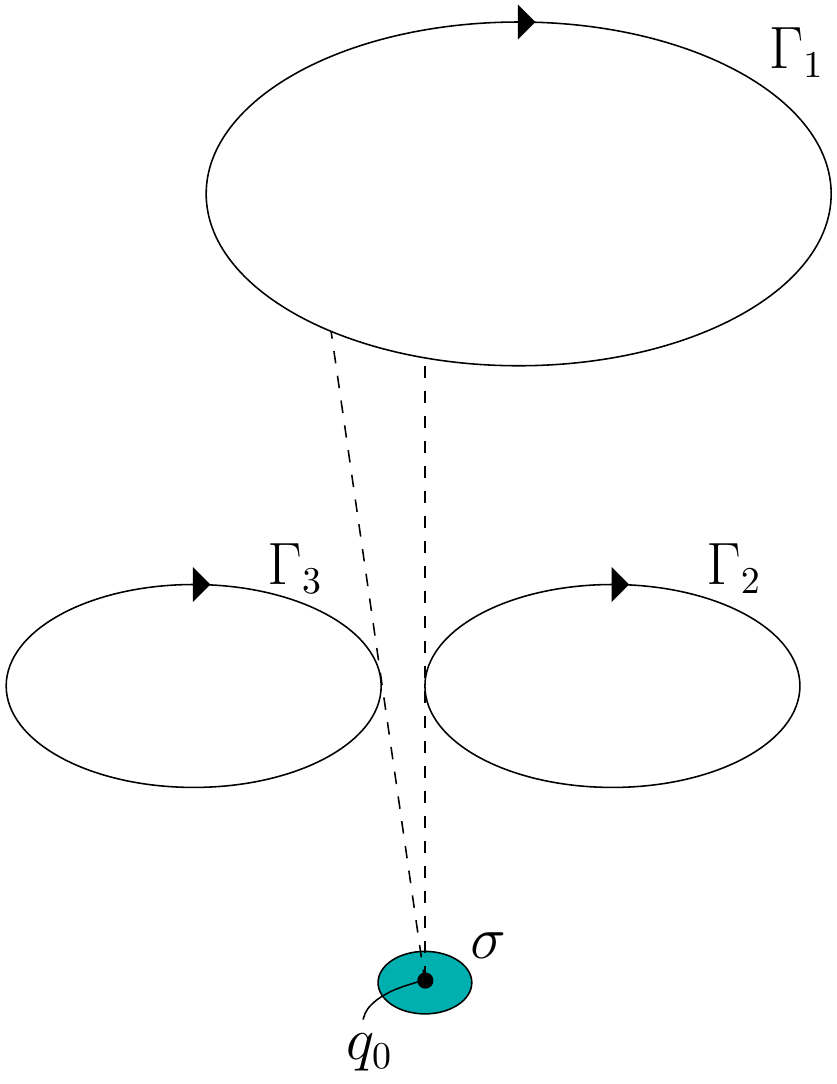}}
\hspace{.2 in}
\resizebox{!}{2.7 in}{
\includegraphics{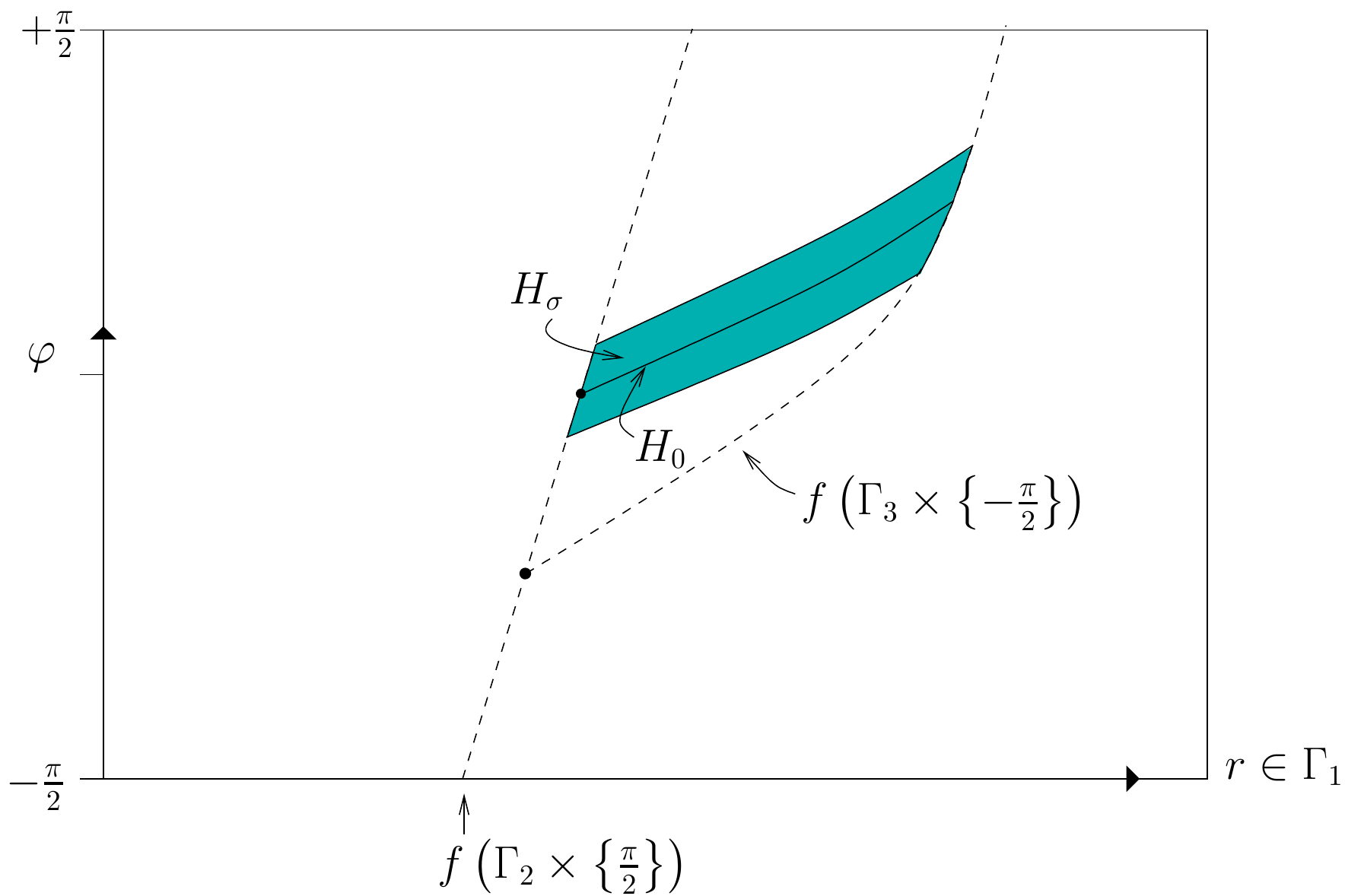}}\label{figure 4}

\caption{Representative examples of the geometry of Type II holes.
\emph{Top set:} On the left is a configuration on the billiard table
$X$, while on the right is the resulting configuration in the subset
$\Gamma_1 \times [-\frac{\pi}{2}, \frac{\pi}{2}]$ of phase space. In
this subset, $H_0$ consists of a single primary segment whose
endpoints lie on $f(\Gamma_2\times \{\frac{\pi}{2}\}) $ and
$\Gamma_1\times \{\frac{\pi}{2}\} $. For $\sigma\ne 0 $, $\partial
H_\sigma$ contains two primary segments and a single secondary
segment that lies on  $f(\Gamma_2\times \{\frac{\pi}{2}\}) $.
(Recall that by convention $\partial H_\sigma $ does not include
subsegments of $\partial M$.)
 \emph{Bottom set:} The analogous situation when the view of $\Gamma_1$ from
$q_0 $ is obstructed by two scatterers, instead of just one. Observe
that now $\partial H_\sigma$ contains two secondary segments in
$\Gamma_1 \times [-\frac{\pi}{2}, \frac{\pi}{2}]$. The situation
when the view of $\Gamma_1$ from $q_0 $ is unobstructed by other
scatterers is simple and is left to the reader.}

\end{figure}

To prove the lemma, observe first that
 $H_0$ can have no more than one component in any
connected component of $M \setminus f(\partial M)$. Second, $\omega$
must also be entirely contained inside one connected component of $M
\setminus f(\partial M)$. This is because unstable curves for $f$
cannot cross the discontinuity set of $f^{-1}$. As a consequence,
$\omega$ also cannot cross any secondary segment as secondary
segments of $\partial H_\sigma$ are contained in $f(\partial M)$.

It remains to show that $\omega$ can meet each primary segment in at
most one point. Although primary segments are increasing, their
tangent vectors lie outside of unstable cones (except at $\partial
M$ where the unstable cone is degenerate). This is because the
curves in $\partial B_\sigma \setminus f^{-1}\partial M$ are
decreasing, while the unstable cones are defined to be the forward
images of $\{0 \le \frac{d \varphi}{dr} \le \infty\}$ under $Df$.
Hence primary segments have greater slopes than $\omega$.
\end{proof}

As pointed out in Sect. 2.3, item 4, Lemma 3.3 must be modified to
account for the deletions that arise from intersections with forward
images of $\Omega_n$, and one might be concerned about the absence
of uniform estimates on transversality in ($\dagger$) between
$\partial H_\sigma$ and unstable curves. This, in fact, is not a
problem, because such deletions occur only in neighborhoods of
$f^{-1}\partial M_\sigma$, which are decreasing curves and hence
uniformly transversal to $u$-curves.

This completes the modifications associated with item 4.

\medskip
No modifications are required for items 5, 7, 8 and 9.

\medskip
\noindent {\bf Modifications in Item 6.} In the proof of (**)', the
argument needs to be modified, again due to the difference in
geometry: In order to prove that $Q^* \cap (\cup_0^q f^i(\partial
M_0)) = \emptyset$ (Sublemma B), in the case of Type I holes we use
that $\cup_1^q f^i(\partial M_0)$ is the union of finitely many
piecewise smooth increasing curves that stretch from $\{\varphi =
-\frac{\pi}{2}\}$ to $\{\varphi = \frac{\pi}{2}\}$. For Type II
holes, this is not true. However, it can be arranged that Sublemma B
will continue to hold as we now explain: First,
$$
\cup_1^q f^i(\partial M_0) \ \subset \left((\cup_1^qf^i(\partial M))
\cup (\cup_0^{q-1} f^i(H_0)) \right) \ \cup \ f^q(H_0) \ .
$$
If we write the right side as $A \cup f^q(H_0)$, then $A$ has the
desired geometry, i.e. it is the union of finitely many piecewise
smooth increasing curves that stretch from $\{\varphi =
-\frac{\pi}{2}\}$ to $\{\varphi = \frac{\pi}{2}\}$. Thus the same
argument as before shows that this set is disjoint from $Q^*$. One
way to ensure that $Q^* \cap f^q(H_0) = \emptyset$ is to choose the
mixing boxes $U_j$ disjoint from $H_0$, which can easily be arranged
given the geometry of primary segments discussed above.

\medskip
This completes the proof of Proposition 2.2 for Type II holes.

%%%%%%%%%%%%%%%%%%%%%%%%%%%%%%%%%
%%%%%%%%%%%%%%%%%%%%%%%%%%%%%%%%%

\section{Escape Dynamics on Markov Towers}
\label{towers}

In this section and the next, we lift the problems from the billiard systems
in question to their Markov tower extensions, and solve the problems there.
In Sect.~4, we review relevant works and formulate results on towers.
Proofs are given in Sect.~5.

\subsection{From generalized horseshoes to Markov towers (review)}

It is shown in [Y] that given a map $f: M \to M$ with a generalized horseshoe
$(\Lambda, R)$ as defined in Sect.~2.1, one can associate a Markov extension
$F: \Delta \to \Delta$ which focuses on the return dynamics to $\Lambda$
(and suppresses details between returns) .
We first recall some facts about this very general construction, taking
the opportunity to introduce some notation.

Let
\[
\Delta = \{ (x,n) \in \Lambda \times \N :  n < R(x) \} ,
\]
and define $F: \Delta \to \Delta$ as follows: For
$\ell <R(x)-1$, we let $F(x,\ell) = (x, \ell+1)$, and define
$F(x,R(x)-1) = (f^{R(x)}(x),0)$. Equivalently, one can view
$\Delta$ as the disjoint union
$\cup_{\ell \ge 0} \Delta_\ell$ where
$\Delta_\ell$, the $\ell^{\mbox{\tiny th}}$ level of the tower, is a copy of
$\{x \in \Lambda: R(x)> \ell\}$. This is the representation we will use.
There is a natural projection $\pi: \Delta \to M$ such that $\pi \circ F = f \circ \pi$.
In general, $\pi$ is not one-to-one, but for each $\ell \ge 0$,
it maps $\Delta_\ell$ bijectively onto $f^\ell(\Lambda \cap \{R \ge \ell\})$.

In the construction of $(\Lambda,R)$, one usually introduces an increasing
sequence of partitions of $\Lambda$ into $s$-subsets representing
distinguishable itineraries in the first $n$ steps.
(In Sects.~\ref{young horseshoe} and \ref{proof of horseshoe}, these partitions
were given by $\tilde \pa_\ell$ of $\tilde \Omega_\ell$.) These partitions
induce a partition $\{ \Delta_{\ell, j} \}$ of $\Delta$ which is finite on
each level $\ell$ and and is a (countable) Markov partition for $F$.
We define a separation time $s(x,y) \leq s_0(x,y)$
by $\inf \{ n>0: \mbox{$F^nx, F^ny$ lie in
different $\Delta_{\ell,j}$} \}$.

We borrow the following language from $(\Lambda, R)$ for use
on $\Delta$: For each $\ell,j$, recall that $\Gamma^s(\pi(\Delta_{\ell, j}))$ and
$\Gamma^u(\pi(\Delta_{\ell, j}))$ are the stable and unstable families
defining the hyperbolic product set $\pi(\Delta_{\ell, j})$. We will say
$\tilde \gamma \subset \Delta_{\ell, j}$ is an {\it unstable leaf} of
$\Delta_{\ell, j}$ if $\pi(\tilde \gamma) = \gamma \cap \pi(\Delta_{\ell, j})$
for some $\gamma \in \Gamma^u(\pi(\Delta_{\ell, j}))$, and use
$\Gamma^u(\Delta_{\ell, j})$ to denote the set of all such $\tilde \gamma$.
Let $\Gamma^u(\Delta) = \cup_{\ell,j}
\Gamma^u(\Delta_{\ell, j})$ be the set of all unstable leaves of $\Delta$.
{\it Stable leaves} of $\Delta_{\ell, j}$ and the families $\Gamma^s(\Delta_{\ell, j})$
and $\Gamma^s(\Delta)$ are defined similarly.

\medskip
Associated with $F:\Delta \to \Delta$, which we may think of as a ``hyperbolic
tower", is its quotient ``expanding tower" obtained by collapsing stable leaves
to points. Topologically, $\bDelta = \Delta/\! \!\sim$ where for
$x,y \in \Delta$, $x \sim y$ if and only if $y \in \gamma (x)$ for some
$\gamma \in \Gamma^s(\Delta)$. Let $\bpi
: \Delta \to \bDelta$ be the projection defined by $\sim$, and let
$\barF: \bDelta \to \bDelta$ be the induced map on $\bDelta$
satisfying $\barF \circ \bpi = \bpi \circ F$. We will use the notation
 $\bDelta_\ell =\bpi (\Delta_\ell), \bDelta_{\ell,j} = \bpi(\dlj)$, and so on.

It is shown in [Y] that there is a well defined differential structure on $\bDelta$
preserved by $\barF$. Recall that $\mu_\gamma$ is the Riemannian measure on
$\gamma$, and for $\gamma, \gamma'  \in \Gamma^u(\Lambda)$,
$\Theta_{\gamma, \gamma'}:
\gamma \cap \Lambda \to \gamma' \cap \Lambda$ is
the holonomy map obtained by sliding along stable curves, i.e.
$\Theta_{\gamma, \gamma'}(x) = \gamma^s(x) \cap \gamma'$.
We introduce the following notation: For $x \in \Lambda_i \cap \gamma$,
let $\gamma'$ be such that $f^{R_i}(\gamma) \subset \gamma'$. Then
 $J^u(f^R)(x) = J_{m_\gamma, m_{\gamma'}}(f^{R_i}|
(\gamma \cap \Lambda_i))(x)$ is the Jacobian of $f^R$ with respect to
the measures $m_\gamma$ and $m_{\gamma'}$.
Lemma 1 of [Y], which we recall
below, is key to the differential structure on $\bDelta$.

\begin{lemma}
\label{lemma:jacobian}
There is a function $u: \Lambda \to {\mathbb R}$ such that for each
$\gamma \in \Gamma^u(\Lambda)$, if $m_\gamma$ is the measure whose
density with respect to $\mu_\gamma$ is $e^u I_{\gamma \cap \Lambda}$,
then we have the following:
\begin{itemize}
  \item[(1)] For all $\gamma$, $\gamma' \in \Gamma^u(\Lambda)$,
  $(\Theta_{\gamma,\gamma'})_*m_\gamma = m_{\gamma'}$.
   \item[(2)] $J^u(f^R)(x) = J^u(f^R)(y)$ for all $y \in \gamma^s(x)$.
  \item[(3)] $\exists C_1>0$ (depending on $C$ and $\alpha$) such that
  for each $i$ and all $x,y \in \Lambda_i \cap \gamma$,
    \begin{equation}\label{eq:C_1}
        \left| \frac{J^u(f^R)(x)}{J^u(f^R)(y)} - 1\right|
    \leq C_1 \alpha^{s(f^Rx, f^Ry)/2} .
    \end{equation}
 \end{itemize}
 The properties of $u$ include $|u| \le C$ and $|u(x)-u(y)| \le
 4C \alpha^{\frac12 s(x,y)}$ on each $\gamma$.
\end{lemma}

(1) and (2) together imply that there is a natural measure $\bm$ on
$\bDelta$ with respect to which  the Jacobian of $\barF$,
$J\barF$, is well defined: First, identify $\bDelta_0$ with $\gamma \cap
\Lambda$ for {\it any} $\gamma \in \Gamma^u(\Lambda)$, and let
$\bm|_{\Delta_0}$ be the measure that corresponds to $m_\gamma$.
(1) says that $\bm$ so defined is independent of $\gamma$, and (2) says that
with respect to $\bm$, $J\barF^R(x) = J^u(f^R)(y)$ for any
$y \in \gamma^s(x)$. We then extend $\bm$ to $\cup_{\ell >0}
\bDelta_\ell$ in such a way that $J\barF \equiv 1$ on all of $\bDelta \setminus
\barF^{-1}(\bDelta_0)$.

In the rest of Sect.~4.1 we will assume $\bm\{R>n\} < C_0 \theta_0^n$
for  some $C_0 \ge 1$ and $\theta_0<1$.\footnote{Our default rule is to use
the same symbol for corresponding objects for $f, F$ and $\barF$ when
no ambiguity can arise given context. Thus $R$ is the name of the return time
function on $\Lambda, \Delta_0$ and $\bDelta_0$.}

One of the reasons for passing from the hyperbolic tower to the expanding
tower is that the spectral properties of the transfer operator associated with
the latter can be leveraged. We fix $\beta$ with $1>\beta > \max \{ \theta_0, \sqrt{\alpha} \}$, and define a symbolic metric on $\bDelta$
 by $d_\beta(x,y) = \beta^{s(x,y)}$.
Since $\beta > \sqrt{\alpha}$, Lemma~\ref{lemma:jacobian}(3) implies that
$J\barF$ is log-Lipshitz with respect to this metric.
A natural function space on $\bDelta$
is ${\cal B} = \{ \rho \in L^1(\bDelta, \bm): \|\rho\| < \infty \}$
where $\|\rho\| = \|\rho\|_\infty + \|\rho\|_{\lip}$ and
\[
\|\rho\|_\infty = \sup_{\ell, j} \sup_{x \in \bDelta_{\ell,j}}
                      |\rho(x)| \beta^\ell ,  \qquad
\|\rho\|_{\lip} = \sup_{\ell,j} \mbox{Lip}(\rho |_{\bDelta_{\ell,j}}) \beta^\ell \ .
\]
Lip$(\cdot)$ above is with respect to the symbolic metric $d_\beta$.
The weights $\beta^\ell$ provide the needed contraction from
one level to the next, and $\beta > \theta_0$ is needed to maintain exponential
tail estimates.

%%%%%%%%%%%%%%%%%%%%%%%%%%%%%%%%%%%%%
\subsection{Towers with Markov holes}
\label{markov holes}

Now consider a leaky system $(f,M,H)$ as defined in Sect.~2.1, and
suppose $(\Lambda, R)$ is a generalized horseshoe respecting the hole $H$.
Let $F:\Delta \to \Delta$ be the associated tower map with $\pi: \Delta \to M$,
and let $\tilde H = \pi^{-1}(H)$. Then $(F, \Delta, \tilde H)$ is a leaky system in itself.
With the horseshoe respecting $H$, we have that $\tilde H$ is the union
of a collection of $\dlj$, usually an infinite number of them; we refer to holes of this type
as ``Markov holes". The notation $H_\ell := \tilde H \cap \Delta_\ell$ will be
used. Projecting and letting $\bH = \bpi(\tilde H)$, we obtain the quotient
leaky system $(\barF, \bDelta, \bH)$.
Let us say $(F, \Delta, \tilde H)$ and
$(\barF, \bDelta, \bH)$ are {\it mixing} if
the surviving dynamics of the horseshoe that gives rise to these towers are mixing;
see Sect.~2.1.

Letting $\hDelta = \Delta \setminus \tilde H$, we introduce the notation
$$\Delta^n = \cap_{i=0}^n F^{-1}\hDelta
= \{ x \in \Delta: F^ix \notin \tilde H \mbox{ for } 0 \le i \le n \}\ ,$$
so that in particular $\hDelta = \Delta^0$.
Corresponding objects for $(\barF, \bDelta, \bH)$ are denoted by $\bDelta^n$
.
%%%%%%%%%%%%%%%%%%%%%%%
\subsubsection{What is known: Spectral properties of expanding towers}

Expanding towers (that are not necessarily quotients of hyperbolic towers)
with Markov holes were studied in \cite{demers exp} and
\cite{demers bruin}. The following theorem summarizes several results proved
in \cite[Proposition 2.4, Corollary 2.5]{demers bruin}, under some conditions
on the tower that are easily satisfied
here. We refer the reader to \cite{demers bruin} for detail,
and state their results in our context of $(\barF, \bDelta, \bH)$.

Let $\B = \{\rho \in L^1(\bDelta^0, \bm): \|\rho\|< \infty\}$
where $\|\rho\|$ is
as above, and let $\bLp$ denote the transfer operator associated with
$\barF|_{\bDelta^1}$ defined on $\B$, i.e., for $\rho \in \B$
and $x \in \bDelta^0$,
\[
\bLp \rho(x) = \sum_{y \in \bDelta^0 \cap \barF^{-1}x} \rho (y) (J\barF(y))^{-1} .
\]

\begin{theorem}{\cite{demers bruin}}
\label{thm:exp conv}
Let $(\barF, \bDelta, \bH)$ be such that (i) $(\barF, \bDelta)$ has exponential
return times and (ii) $(\barF, \bDelta, \bH)$ is mixing. Assume the following
condition on hole size:
\begin{equation}
\sum_{\ell \geq 1} \beta^{-(\ell-1)}\bm(H_\ell)
     < \frac{(1-\beta)\bm(\bDelta_0)}{1+C_1}\ .
\label{holesize}
\end{equation}
Then the following hold:
\begin{enumerate}
\item[(1)] $\bLp$ is quasi-compact with a unique eigenvalue $\ra_*$ of maximum
modulus; $\ra_*$ is real and $>\beta$, and it has a unique eigenfunction
$h_* \in \B$ with $\int h_*d\bm=1$. In addition, there exist constants
$D>0$ and $\tau<1$ such that for all $\rho \in \B$,
\[
\| \ra_*^{-n}\bLp^n\rho - d(\rho)h_*\| \leq D\|\rho \| \tau^n, \; \; \;
\mbox{where }
d(\rho) = \lim_{n\to \infty} \lambda^{-n} \int_{\bDelta^n} \rho \, d\bm < \infty .
\]
\item[(2)] The eigenvalue $\ra_*$ satisfies
$\ra_* > 1 - \frac{1+C_1}{\bm(\bDelta_0)}
\sum_{\ell \geq 1} \beta^{-(\ell-1)}\bm(H_\ell)$.
\end{enumerate}
\end{theorem}

The spectral property of $\bLp$ as described in Theorem \ref{thm:exp conv}(1)
 implies
that all $\rho$ except for those in a codimension 1 subspace have
$d(\rho) \ne 0$. Given the pivotal role played by the base $\bDelta_0$ of the
tower $\bDelta$,
one would guess that for a density $\rho$, if $\rho>0$ on $\bDelta_0$, then
$d(\rho) \ne 0$. A slightly more general condition is given
in Corollary \ref{cor:positive} below.
 We call $\bDelta_{\ell,j}$ a \emph{surviving element} of the tower if some part of $\bDelta_{\ell,j}$
returns to $\bDelta_0$ before entering $\overline H$.

\begin{corollary} \cite{demers bruin}
\label{cor:positive}
Let $\rho \in \B$ be a nonnegative function that is $>0$ on a surviving
$\bDelta_{\ell,j}$. Then $d(\rho) >0$.
\end{corollary}

%%%%%%%%%%%%%%%%%%
\subsubsection{What is desired: Results for hyperbolic towers}

Here we formulate a set of results for the hyperbolic tower
that connect the results in Sect.~4.2.1 to the stated theorems for billiards.
Let $\tG$ be the class of measures $\eta$ on $\Delta$ with
the following properties:
(i) $\eta$ has absolutely continuous conditional measures
on unstable leaves; and (ii) $\bpi_* \eta = \brho d\bm$ for some
$\brho \in \B$ with $d(\brho) >0$.

Let $(\Lambda^{(\sigma)}, R^{(\sigma)})$ be a generalized horseshoe with
the properties in Proposition 2.2, and let
$(F,\Delta)$ be its associated tower.
 Let $n(\Delta, \tilde H) := \sup\{\ell: H_\ell = \emptyset\}$,
i.e., $n(\Delta, \tilde H) = n(\Lambda^{(\sigma)}, R^{(\sigma)}; H_\sigma)$
as defined in Sect.~2.2.

\begin{theorem}
\label{thm:tower}
Assume that $n(\Delta, \tilde H)$
is large enough that
\begin{equation}
\sum_{\ell \geq n(\Delta, \tilde H)} \beta^{-(\ell-1)}\bm(\bDelta_\ell)
     < \frac{(1-\beta)\bm(\bDelta_0)}{1+C_1}\ .
\label{holelevel}
\end{equation}
Then the following hold:
\begin{itemize}
\item[(a)] There exists $\ra_*<1$ such that for all $\eta \in \tG$,
\[
\log \ra_* = \lim_{n\to \infty} \frac{1}{n} \log \eta (\Delta^n)
\ .
\]
\item[(b)] There exists a conditionally
invariant distribution $\tmu_* \in \tG$ with escape rate $- \log \ra_*$
for which the following hold:
For all $\eta \in \tG$, if $\brho$ is the density of $\bpi_*\eta$ and
$d(\brho)$ is as in Theorem \ref{thm:exp conv}, then
\[
\lim_{n \to \infty} \frac{\F_*^n\eta}{\F_*^n\eta(\hDelta)} = \tmu_*
\qquad \mbox{and} \qquad
\lim_{n \to \infty} \ra_*^{-n} \F_*^n\eta = d(\brho) \cdot \tmu_*
\]
where the convergence  is in the weak* topology.
\item[(c)] $\tmu_*$ has absolutely continuous conditional measures
on unstable leaves.
\end{itemize}
\end{theorem}

\begin{remark}
{\rm In Sect. 5.1 we show that $\bpi_* \tmu_* = h_* \bm$.  Thus $\tmu_* \in \tG$
and the $\vartheta_*$ of Theorem~\ref{thm:exp conv} is the same as the
$\vartheta_*$ of Theorem~\ref{thm:tower}.}
\end{remark}

Theorem \ref{thm:tower} treats one hole at a time. The
following uniform bounds are also needed, mostly for purposes of
proving Theorem 4.

\begin{proposition}
\label{prop:uniform bounds}
Consider all $(F,\Delta, \tilde H)$ arising from
any $(\Lambda^{(\sigma)}, R^{(\sigma)})$ in Proposition 2.2 for which
the hole condition in (\ref{holelevel}) is met.
Let $\tmu_*$ and $\ra_*$ be as in Theorem~\ref{thm:tower}. Then there are
constants $C_2, K >0$ such that  \vspace{-.1 in}
\begin{enumerate}
\item[(i)] the conditional densities $\rho_\gamma$
of $\tmu_*|_{\Delta_\ell}$ with respect to $\mu_\gamma$
on unstable leaves
satisfy $C_2^{-1} \ra_*^{-\ell} \leq \rho_\gamma \leq C_2 \ra_*^{-\ell}$;
\vspace{-.1 in}
\item[(ii)]  $\tmu_*(\cup_{\ell> L} \Delta_\ell)\leq  K\beta^{-L}\theta_0^L$; and
\vspace{-.1 in}
\item[(iii)] $\ra_* \to 1$ as $n(\Delta, \tilde H) \to \infty$.
\end{enumerate}
\end{proposition}

%%%%%%%%%%%%%%%%%%%%%%%%%%%%%%%%%%%%%%%%%%%%%%%%%%%%%%%%%%%%%%%%%%

\section{Proofs of Theorems on the Tower}
\label{tower proofs}

The following notational abbreviations {\it are used only
in this section}:

\smallskip
\noindent -- We will sometimes drop the $\ \tilde{} \ $ used to
distinguish between objects on $M$ and corresponding

objects on $\Delta$; there can be
no ambiguity as long as we restrict ourselves to the towers.

\smallskip
\noindent -- We will at times drop the $\mathring \ $ in $\F$. Specifically,
$F^n_*\eta$ is to be interpreted as $\F^n_*\eta$, and $\barF^n_*\bareta$

is to be interpreted the same way.

\medskip
We focus on the stable direction, since that is what lies between
Theorem \ref{thm:exp conv}
and Theorem~\ref{thm:tower}. The following is a class of test functions on $\hDelta$
that are Lipschitz in the stable
direction. For $\gamma^s \in \Gamma^s(\Delta)$ and $x,y \in \gamma^s$,
we denote by $d^s(x, y)$ the distance between $\pi(x)$ and $\pi(y)$ according to
the $p$-metric, so that $d^s(F^nx, F^ny) \le \lambda^{-n} d^s(x,y)$
for some $\lambda>1$ (see Sect.~2.3).
Let $\mathcal{F}_b$ be the set of bounded, measurable functions on
$\hDelta$. For $\vf \in \mathcal{F}_b$,
we define $|\vf|^s_{\lip}$ to be the Lipshitz constant
of $\vf$ restricted to stable leaves, i.e.
\[
|\vf|^s_{\lip} = \sup_{\gamma^s\in \Gamma^s(\hDelta)} \sup_{x,y \in \gamma^s}
            \frac{\vf(x)-\vf(y)}{d_s(x,y)}  ,
\]
and let Lip$^s(\hDelta) = \{ \vf \in \mathcal{F}_b : |\vf|^s_{\lip} < \infty \}$.

%%%%%%%%%%%%%%%%%%%%%%%%%%%%%%%

\subsection{Proof of Theorem~\ref{thm:tower}}

\noindent {\bf A. Escape rates}

\smallskip
Theorem~\ref{thm:tower}(a) follows easily from Theorem \ref{thm:exp conv}
as $(F, \Delta, H)$ and $(\barF, \bDelta, \bH)$ have the same escape rate.
In more detail, let $\eta \in\tG$ and notice that since $H$ is a union of $\dlj$,
we have, for each $n$,
$\eta(\Delta^n) = \bareta (\bDelta^n)$ where $\bareta = \bpi_* \eta$.
By definition of $\tG$, $\frac{d\bareta}{d\bm} = \brho \in \B$ with $d(\brho)>0$.
Theorem~\ref{thm:exp conv}(1) then implies that $\ra_*^{-n} \bLp^n\brho$
converges  to $d(\brho) h_*$.  Since the convergence
is in the $\| \cdot \|$-norm, we may integrate with respect to $\bm$.  Noting that
$\int_{\bDelta} \bLp^n \brho \, d\bm =  \int_{\bDelta^n} \brho \, d\bm
= \bareta(\bDelta^n)$, we have
\begin{equation}
\label{eq:escape conv}
\lim_{n \to \infty} \ra_*^{-n} \eta(\Delta^n) =
\lim_{n \to \infty} \ra_*^{-n} \bareta(\bDelta^n) = d(\brho) .
\end{equation}
Thus $-\log \ra_*$, where $\ra_*$ is the eigenvalue in Theorem \ref{thm:exp conv},
is the common escape rate of $(F, \Delta, H)$ for initial distributions
 in $\tG$.

\bigskip
\noindent {\bf B. Uniqueness of limiting distributions}

\smallskip
We first prove uniqueness postponing the proof of existence of limiting
distributions.

Given $\eta \in \tG$, we define a measure $\eta^s$ on
$\Gamma^u(\Delta)$, i.e. a measure transverse to unstable leaves, as follows:
Set $\eta^s(\Gamma^u(\dlj))=0$ if $\eta(\dlj)=0$.
If $\eta(\dlj) \ne 0$, then $\eta^s|_{\Gamma^u(\dlj)}$ is the factor measure
of $\eta|_{\dlj}$  {\it normalized},
and $\{ \rho dm_\gamma, \gamma \in \Gamma^u(\dlj)\}$ is the disintegration of
$\eta$ into measures on unstable leaves. We will use the convention that
$\eta^s(\dlj)=1$, and $ \rho|_\gamma$ is the density with respect to
$m_\gamma$, so that $\int \rho|_\gamma d\eta^s(\gamma) = \brho$ where
$d \bpi_*\eta = \brho d\bm$.

\begin{lemma}
\label{lemma:unique} Let $\eta_1$ and $\eta_2 \in \tG$. Suppose for $i=1,2$,
there exists $\mu_*^i$ such that
$$\lim_{n\to \infty}\ra_*^{-n} F_*^n\eta_i = d(\brho_i) \mu_*^i$$ where
$\brho_i$ is the density of $\bpi_* \eta_i$ Then $\mu_*^1=\mu_*^2$.
\end{lemma}

The crux of the argument for Lemma \ref{lemma:unique} is contained in

\begin{lemma}
\label{lemma:close convergence}
Let $\eta_1$ and $\eta_2$ be as above, and assume $\brho_1=\brho_2$.
Then for all $\vf \in \mbox{Lip}^s(\hDelta)$,
$\ra_*^{-n} |F^n_*\eta_1(\vf) - F^n_*\eta_2(\vf)| \to 0$ exponentially fast as $n \to \infty$.
\end{lemma}

\begin{proof}
For $i=1,2$, let $\eta^s_i$ and $\rho_i$ be the (normalized)
factor measure and (unnormalized) densities on $\gamma \in \Gamma^u(\Delta)$
of $\eta_i$ as described above.

We consider functions which are constant along
stable leaves to be defined on both $\hDelta$ and $\bDelta^0$ and do not distinguish between the two versions of such functions.
For each $\dlj$, let $\hat{\gamma} \in \Gamma^u(\dlj)$ be
a representative leaf. Then
\begin{equation}
\label{eq:factor diff}
|F_*^n \eta_1 (\vf) - F_*^n \eta_2 (\vf) |
\; \leq \; \sum_{\ell,j}
        \int_{\hat{\gamma} \cap \Delta^n_{\ell,j}} dm_{\hat{\gamma}}
        \left| \int_{\gamma^s} \rho_1 \, \vf \circ F^n d\eta^s_1
                -  \int_{\gamma^s} \rho_2 \, \vf \circ F^n d\eta^s_2 \right|  .
\end{equation}
Next fix $x \in \hat{\gamma} \cap \Delta^n$ and estimate the integrals on $\gamma^s(x)$.
Define $\bvf_n = \int_{\gamma^s} \vf \circ F^n \, d\eta^s_1$. Then,
\[
\begin{split}
\left| \int_{\gamma^s} \rho_1 \, \vf \circ F^n d\eta^s_1
                -  \int_{\gamma^s} \rho_2 \, \vf \circ F^n d\eta^s_2 \right|
\leq & \left| \int_{\gamma^s} \rho_1 \, (\vf \circ F^n - \bvf_n) d\eta^s_1 \right|
+ \left| \int_{\gamma^s} \rho_2 \, (\vf \circ F^n - \bvf_n) d\eta^s_2 \right| \\
& + \left| \int_{\gamma^s} \bvf_n \rho_1 \, d\eta^s_1
                 -  \int_{\gamma^s} \bvf_n \rho_2 \, d\eta^s_2 \right| \ .
\end{split}
\]
Since $\bvf_n$ is constant on $\gamma^s$ and $\brho_1=\brho_2$,
the third term above is 0.  For the first
two terms, we note that for each $y \in \gamma^s(x)$,
$|\bvf_n(y) - \vf \circ F^n(y)| \leq |\vf |^s_{\lip} \lambda^{-n}$. Thus
\begin{equation}
\label{eq:close proj}
\ra_*^{-n}|F_*^n \mu_1 (\vf) -  F_*^n \mu_2 (\vf) |
\leq \ra_*^{-n} \sum_{\ell,j} \int_{\bDelta^n_{\ell,j}} 2 \brho_1 d\bm \,
 | \vf |^s_{\lip} \lambda^{-n}
= 2 \ra_*^{-n} | \bLp^n \brho_1|_1 | \vf |^s_{\lip} \lambda^{-n} ,
\end{equation}
which proves the lemma since $\ra_*^{-n} |\bLp^n \brho_1| \to d(\brho_1)$
by Theorem~\ref{thm:exp conv}.
\end{proof}

\begin{remark} {\rm
We have used in the proof above a property of the billiard maps,
namely $d^s(F^nx, F^ny) \le \lambda^{-n} d^s(x,y)$. For general towers,
one has only the contraction guaranteed by {\bf (P3)} which is nonuniform. It is not
hard to see that the lemma holds in the more general case with the exponential
rate given by $\max \{ \alpha^{\frac{n}{2}}, \beta^{-n}\theta_0^n \}$
in the place of $\lambda^{-n}$;
we leave the proof to the interested reader.}
\end{remark}

\begin{proof}[Proof of Lemma \ref{lemma:unique}]
Let $\bmu_* = h_* \bm$ be the
conditionally invariant measure given by
Theorem~\ref{thm:exp conv}.  For $i=1,2$, we have, on the one hand,
\[
\lim_{n\to \infty} \ra_*^{-n} \barF_*^n \bareta_i = d(\brho_i) \bmu_* ,
\]
which follows from Theorem~\ref{thm:exp conv}, and on the other,
\[
\lim_{n\to \infty} \ra_*^{-n} \bpi_* F^n_* \eta_i = d(\brho_i) \bpi_* \mu_*^i,
\]
which follows from the hypothesis of the lemma.
 Since $\bpi_* F^n_* \eta_i = \barF^n_* \bpi_* \eta_i$ for each $n \geq 0$,
we have
$\bpi_* \mu_*^1 = \bmu_* = \bpi_* \mu_*^2$.
Thus $\ra_*^{-n} |F_*^n \mu_*^1 - F_*^n \mu_*^2| \to 0$ as $n \to \infty$ by
Lemma~\ref{lemma:close convergence}.
But $\ra_*^{-n} F_*^n \mu_*^i = \mu_*^i$ since $\mu_*^i$ is conditionally invariant.
Hence $\mu_*^1 = \mu_*^2$.
\end{proof}

\noindent {\bf C. Convergence to conditionally invariant measure}

\smallskip
For a probability measure $\eta$ on $\hDelta$,
$|F^n_*\eta| = \eta(\Delta^n) = \bpi_* \eta(\bDelta^n)$.
So for $\eta \in \tG$, \eqref{eq:escape conv} implies
$\lim_{n\to \infty} \ra_*^{-n} |F^n_*\eta| = d(\brho) >0$
where $\brho$ is the density of $\bareta = \bpi_* \eta$.
More than that is true:

\begin{lemma}
\label{lemma:cauchy}
$\ra_*^{-n} F_*^n \eta/d(\brho)$ converges weakly to a conditionally invariant
probability measure $\mu_*$ as $n \to \infty$.
\end{lemma}

This is half of Theorem~\ref{thm:tower}(b). Once we have this, it will follow
immediately that
\begin{equation}
\label{eq:alt conv}
\lim_{n\to \infty} \frac{F^n_*\eta}{|F^n_*\eta|} =
\lim_{n\to \infty} \frac{\ra_*^{-n} F^n_*\eta}{\ra_*^{-n} |F^n_*\eta|} = \mu_* \ ,
\end{equation}
which is the other half.

We will use the following algorithm to
``lift" measures from $\bDelta$ to $\Delta$: Fix a measure
$\mu^s$ on $\Gamma^u(\Delta)$ with $\mu^s(\Gamma^u(\dlj))=1$.
Given $\bareta$ on $\bDelta$ with density $\brho$, we define
$\bpi^{-1}_* \bareta$ to be the measure on $\Delta$ with the property
that restricted to each $\dlj$, $\bpi^{-1}_* \bareta$ decomposes into
the factor measure $\mu^s$ and leaf measures $\{\rho dm_\gamma\}$ where
$\rho|_{\bpi^{-1}(x)} \equiv \brho(x)$.
Notice that $\bpi_* \bpi_*^{-1} \bareta = \bareta$.

\begin{proof}[Proof of Lemma \ref{lemma:cauchy}]
Our first step is to fix $\vf \in \mbox{Lip}^s(\hDelta)$ and show that
$\ra_*^{-n}F_*^n \eta(\vf)$ is a Cauchy sequence.
For a fixed $\mu^s$ as above, and let
$\bvf_n(x) = \int_{\gamma^s(x)} \vf \circ \F^n \, d\mu^s$.
Define $\bareta = \bpi_* \eta$.  Since $\eta \in \tG$, $\bareta$ has density
$\brho \in \B$ with $d(\brho)>0$.
Then by definition of $\bpi_*^{-1}$,
\begin{equation}
\label{eq:average}
(\bpi_*^{-1} \bpi_* \eta) (\vf \circ \F^n) = \sum_{\ell,j}
\int_{\Gamma^u(\dlj)} d\mu^s(\gamma) \int_{\gamma^u} \vf\circ \F^n \, \brho \,
    dm_\gamma
  = \sum_{\ell,j} \int_{\bDelta_{\ell,j}} \brho \, \bvf_n \, d\bm
  = \bpi_*\eta(\bvf_n) .
\end{equation}
For $n, k_1, k_2 \geq 0$, write
\begin{equation}
\label{eq:split}
\begin{split}
 | \ra_*^{-n-k_1}F_*^{n+k_1}  \eta(\vf) - \ra_*^{-n-k_2} & F_*^{n+k_2}\eta(\vf) |
\leq \ra_*^{-n-k_1}| F_*^{n+k_1}\eta(\vf) - F_*^n \bpi_*^{-1} \bpi_* F_*^{k_1}\eta(\vf) | \\
& + | \ra_*^{-n-k_1}F_*^n \bpi_*^{-1} \bpi_* F_*^{k_1}\eta(\vf)
- \ra_*^{-n-k_2} F_*^n \bpi_*^{-1} \bpi_* F_*^{k_2}\eta(\vf) | \\
& + \ra_*^{-n-k_2} | F_*^n \bpi_*^{-1} \bpi_* F_*^{k_2}\eta(\vf) \eta(\vf) - F_*^{n+k_2}\eta(\vf) |\ .
\end{split}
\end{equation}
The first and third terms of \eqref{eq:split} are estimated using
Lemma~\ref{lemma:close convergence} since
$\bpi_*(\ra_*^{-k_i} F_*^{k_i}\eta)  = \bpi_*(\ra_*^{-k_i} \bpi_*^{-1} \bpi_* F_*^{k_i}\eta)$ for $i = 1,2$.
Thus by Lemma~\ref{lemma:close convergence},
\[
\begin{split}
& \ra_*^{-n-k_i} | F_*^{n+k_i}\eta(\vf) - F_*^n \bpi_*^{-1} \bpi_* F_*^{k_i}\eta(\vf) |
\leq C' d(\brho)  (|\vf|^s_{\lip} + |\vf|_\infty) \zeta^n
\end{split}
\]
for some $C'>0$ and $\zeta < 1$.

We now fix $n$ and estimate the second term of \eqref{eq:split}. Due to
\eqref{eq:average}, for any $k\geq 0$ we have
\[
\begin{split}
\ra_*^{-n-k}F^n_* & \bpi_*^{-1} \bpi_* F^k_* \eta (\vf) = \ra_*^{-n-k}
\bpi_*^{-1} \bpi_* F_*^k \eta (\vf \circ F^n \cdot 1_{\hDelta^n})
= \ra_*^{-n-k} \bpi_* F_*^k \eta ( \overline{\vf}_n
          \cdot 1_{\bDelta^n}) \\
& = \ra_*^{-n-k} \barF_*^k
\bareta ( \overline{\vf}_n \cdot 1_{\bDelta^n})
= \ra_*^{-n-k} \int_{\bDelta^n}
\overline{\vf}_n \cdot \bLp^k \brho  \, d\bm  .
\end{split}
\]
Recalling that $\brho \in \B$ and $d(\brho)>0$
since $\eta \in \tG$, we estimate
\begin{equation}
\begin{split}
 & | \ra_*^{-n-k_1} F_*^n \bpi_*^{-1} \bpi_* F_*^{k_1}\eta(\vf)
- \ra_*^{-n-k_2} F_*^n \bpi_*^{-1} \bpi_* F_*^{k_2}\eta(\vf) |  \\
& \leq   |\vf|_\infty \ra_*^{-n}
  \int_{\bDelta^n}
        \left| \ra_*^{-k_1} \bLp^{k_1} \brho - d(\brho) h_* \right| d\bm
        + |\vf|_\infty \ra_*^{-n}
  \int_{\bDelta^n}
        \left| \ra_*^{-k_2} \bLp^{k_2} \brho - d(\brho) h_* \right| d\bm .
\end{split}
\label{eq:last term}
\end{equation}
Both terms of \eqref{eq:last term} are small: By Theorem~\ref{thm:exp conv}(1),
\[
\ra_*^{-n}  \int_{\bDelta^n}
        \left|\ra_*^{-k} \bLp^k \brho - d(\brho) h_* \right| d\bm
\leq
\ra_*^{-n}
        \left\| \ra_*^{-k} \bLp^k \brho - d(\brho) h_* \right\|
    \int_{\bDelta^n} 1_\beta \, d\bm
 \leq \ra_*^{-n} |\bLp^n 1_\beta|_1
         D \| \brho \| \tau^k
\]
for $k=k_i$,
where $1_\beta(x) = \beta^{-\ell}$ for $x \in \hDelta_\ell$.
Since $1_\beta \in \B$, $\ra_*^{-n} |\bLp^n 1_\beta|_1$
converges to $d(1_\beta)$ as $n\to \infty$.
Thus \eqref{eq:last term} can be made
arbitrarily small by choosing $k_1$ and $k_2$ sufficiently large.

We have shown that $\ra_*^{-n} F^n_*\eta(\vf)/d(\brho)$ is
a Cauchy sequence and therefore
converges to a number $Q(\vf)$.
The functional $Q(\vf) := \lim_{n \to \infty} F_*^n\eta(\vf)/d(\brho)$ is clearly linear
in $\vf$, positive and satisfies
$Q(1) = 1$.  Also $|Q(\vf)| \leq |\vf|_\infty Q(1)$ so that $Q$ extends to a bounded linear functional
on $C_b^0(\hDelta)$, the set of bounded functions which are continuous on each $\hdlj$.

By the Riesz representation theorem, there exists a unique Borel probability measure
$\mu_*$ satisfying $\mu_*(\vf) = Q(\vf)$ for each $\vf \in C_b^0(\hDelta)$
\cite[Section 56]{halmos}.
Also,
\[
d(\brho) \mu_*(\vf \circ \F) = \lim_{n\to \infty} \ra_*^{-n} F_*^n \eta (\vf \circ \F)
= \ra_* \lim_{n\to \infty} \ra_*^{-n-1} F_*^{n+1} \eta(\vf)
= \ra_* d(\brho) \mu_*(\vf)
\]
so that $\mu_*$ is a conditionally
invariant measure for $\F$ with escape rate $-\log \ra_*$.
\end{proof}

This completes the proof of parts (a) and (b) of Theorem \ref{thm:tower}.
To prove part (c), we must show that $\mu_*$ has absolutely continuous
conditional measures on unstable leaves.
Proof of a stronger version of this fact is contained in the
proof of Proposition~\ref{prop:uniform bounds}(i) below.

Notice that from the proof of Lemma~\ref{lemma:unique}, we have
$\bpi_*\mu_* = \bmu_*$
so that the density of $\bpi_*\mu_*$ is precisely $h_*$.  Since $h_* \in \B$
and $d(h_*) =1>0$, we conclude $\mu_* \in \tG$.

%%%%%%%%%%%%%%%%%%%%%%%%%%%%%%%%%%%%%%%%%%%%%%%%%%%%%%%%%%%%%%%%%%%%%%%%%%%%%%%%%

\subsection{Proof of Proposition~\ref{prop:uniform bounds}}

Consider the set of holes $\Sigma_h(q_0)$ for fixed $q_0$ where $h$ is small
enough as required in Proposition~\ref{prop:horseshoe}.
For $\sigma \in \Sigma_h(q_0)$, let $\Delta^{(\sigma)}$ be the tower with holes
induced by the generalized horseshoe, and
let $\mu_*^{(\sigma)}$ be the conditionally invariant measure given by
Theorem \ref{thm:tower}.

\bigskip
\noindent {\it Proof of (i).}
We drop the superscript $(\sigma)$ in what follows and point out that the
constants we use are uniform for all $\sigma \in \Sigma_h(q_0)$ and $h$ sufficiently
small.

Choose $\gamma_0 \in \Gamma^u(\Delta_0)$ and let $\eta_0$ be the measure supported on $\gamma_0$ with uniform density with respect to $\mu_\gamma$.
We claim that $\eta_0 \in \tG$.  It is immediate that $\bpi_* \eta_0$ has
density $\brho = e^{-u}|_{\gamma_0}$ with respect to $\bm$, which is in $\B$ by
Lemma~\ref{lemma:jacobian}.  To see that $d(\brho) > 0$, notice that
the mixing assumption on $(\barF, \bDelta)$ implies that $\bDelta_0$ is
necessarily a surviving partition element.  Since $\brho > 0$ on $\bDelta_0$,
Corollary~\ref{cor:positive} implies that $d(\brho)>0$.

By Theorem~\ref{thm:tower}(b),
$\eta^{(n)} := F^n_*\eta_0/|F^n_*\eta_0|$ converges to $\mu_*$.
Let $\rho^{(n)}_\gamma$ denote the density of $\eta^{(n)}$  with respect
to $\mu_\gamma$ on
$\gamma \in \Gamma^u(\Delta)$.  Notice that inverse branches of $\F^n$ on
$\gamma$ are well defined. For any $x_1, x_2 \in \gamma$, treating one
branch at a time and summing of all branches, we obtain that
\[
\frac{\rho_\gamma^{(n)}(x_1)}{\rho_\gamma^{(n)}(x_2)}
= \frac{\sum_{y_1 \in \F^{-n}x_1} (J_{\mu_\gamma}\F^n(y_1))^{-1}}
       {\sum_{y_2 \in \F^{-n}x_2} (J_{\mu_\gamma}\F^n(y_2))^{-1}}
\leq \sup_{y_1 \in \F^{-n}x_1} \frac{J_{\mu_\gamma}\F^n(y_2)}
    {J_{\mu_\gamma}\F^n(y_1)}
\leq e^C
\]
by Property {\bf (P4)(b)} where $J_{\mu_\gamma} \F^n$ is the Jacobian of $\F^n$
with respect to $\mu_\gamma$.
Since by Proposition~\ref{prop:horseshoe} the constant $C$ is independent of $\sigma$, $x$ and $n$, we have
\begin{equation}
\label{eq:uniform ratio}
e^{-C} \leq \frac{\sup_{x\in \gamma} \rho_\gamma^{(n)}(x)}
{\inf_{x \in \gamma} \rho_\gamma^{(n)}(x)} \leq e^C .
\end{equation}
This estimate plus the minimum length $\kappa$ of $\mu^u(\Lambda)$ given
by Proposition~\ref{prop:horseshoe} yields the desired uniform upper and lower
bounds on the conditional densities of $\eta^{(n)}$ with respect to $\mu_\gamma$
(and hence to $m_\gamma$) on $\Delta_0$.
The uniformity of these bounds in
$n$ implies that they pass to the conditional densities of $\mu_*$
in the limit as $n \to \infty$.
Since $\mu_*$ is conditionally invariant,
$\mu_*|_{\hDelta_\ell} = \ra_*^{-1} \mu_*|_{\F^{-1}\hDelta_\ell}$.  The required
bounds on the densities extend easily to $\hDelta_\ell$ for $\ell >0$.

\bigskip
\noindent {\it Proof of (ii).}
We decompose $\mu_*$ into a normalized factor measure $\mu^s_*$ on
$\Gamma^u(\Delta_\ell)$
and densities $\rho_\gamma$ with respect to $m_\gamma$ on $\gamma \in \Gamma^u(\Delta_\ell)$.  Then
\[
\begin{split}
\mu_*(\cup_{\ell \geq L} \Delta^{(\sigma)}_\ell)
& = \sum_{\ell \geq L} \int_{\Gamma^u(\Delta_\ell^{(\sigma)})} d\mu^s_*
\int_\gamma \rho_\gamma dm_\gamma
\; \leq \; \sum_{\ell \geq L} C_2 \ra_*^{-\ell}  \bm(\bDelta_\ell)
\; \leq \; C_2  \sum_{\ell \geq L} C_0 \theta_0^\ell \beta^{-\ell}\ .
\end{split}
\]
Here we have used Proposition~\ref{prop:uniform bounds}(i) to estimate
$\rho_\gamma$,  Proposition~\ref{prop:horseshoe}
and Lemma~\ref{lemma:jacobian} for the
uniformity of $C_0$ and $\theta_0$, and the fact that $\ra_* > \beta$.
The sum can be made arbitrarily small since $\beta > \theta_0$.

\bigskip
\noindent {\it Proof of (iii).}
Notice that $n(\Delta, \tH) \geq \bar{n}(h)$ by definition of $\bar{n}(h)$
in Sect. 2.2.
From Theorem~\ref{thm:exp conv}(2),
we know that the escape rate $-\log \ra_*$ satisfies
\[
\ra_* > 1 - \frac{1+C_1}{\kappa} \sum_{\ell \geq 1} \beta^{\ell -1}
\bm(\overline H \cap \bDelta_\ell)
> 1 - \frac{1+C_1}{\kappa} \sum_{\ell \geq \bar{n}(h)}
\beta^{\ell-1} C_0 \theta_0^\ell  .
\]
By Proposition~\ref{prop:horseshoe}, $\bar{n}(h) \to \infty$ as $h \to 0$,
so that $\ra_* \to 1$.

%%%%%%%%%%%%%%%%%%%%%%%%%%%%%%%%%%%%%%%%%%%%%%%%%%%%%%%%%%%%%%%%%%%%%%%%%%%%%%%%%

\section{Proofs of Theorems for Billiards}
\label{billiard proofs}

In the Proofs of Theorems 1--3, we fix a hole $\sigma$ that is acceptable with respect
to  Proposition~\ref{prop:horseshoe} and for which
$n(\Lambda^{(\sigma)}, R^{(\sigma)}, H_\sigma)$ is large enough to meet
the condition in Theorem \ref{thm:tower}. We suppress mention of $\sigma$,
and let $(F, \Delta, \tilde H)$ be the tower constructed from
$(\Lambda, R, H)$. Define
 $M^n = \cap_{i=0}^n f^{-n} \hM$, $M^\infty = \cap_{n \ge 0} M^n$.

\subsection{Proof of Theorems~\ref{thm:escape rate} and \ref{thm:accim}}
\label{proof of accim}

The first order of business is to show that each $\eta \in \G$ can be lifted to
a measure $\teta \in \tG$ in such a way that the escape dynamics on $\hDelta$
with initial distribution $\teta$ reflect those on $\hM$ with initial distribution $\eta$.
Recall that the natural invariant probability measure for the closed billiard system
$f: M \to M$ is denoted by $\nu$. In [Y, Sect.~2], it is shown that there is a unique
invariant probability measure $\tnu$ for the tower map $F:\Delta \to \Delta$ with absolutely continuous conditional
measures on unstable leaves, and this measure has the property $\pi_*\tnu = \nu$.
Given $\eta \in \G$, we define $\teta$ on $\Delta$ as follows:
By definition, every $\eta \in \G$ is absolutely continuous with respect to $\nu$.
Let $\psi =  \frac{d\eta}{d\nu}$. We take $\teta$ to be the measure
given by $d\teta = \tilde \psi d\tnu$ where $\tilde \psi = \psi \circ \pi$.
This implies in particular that $\pi_* \teta = \eta$.

\begin{lemma}
\label{lemma:G}
If $\eta \in \G$, then $\teta \in \tG$.
\end{lemma}

As before, let $\mathcal{F}_b$ denote the set of bounded functions on $\Delta$.
For $\vf \in \mathcal{F}_b$ and $\gamma \in \Gamma^u(\Delta)$, we
let Lip$^u(\vf|_\gamma)$ be the Lipschitz constant of $\vf|_\gamma$
with respect to the $d_\beta$-metric (notice that $d_\beta$, the symbolic metric
defined on $\bDelta$, can be thought of as a metric on unstable leaves).
Let
$$|\vf|^u_{\lip} = \sup_{\gamma \in \Gamma^u(\Delta)} {\rm Lip}^u(\vf|_\gamma)\ ,
$$
and Lip$^u(\Delta) = \{ \vf \in \mathcal{F}_b : |\vf|^u_{\lip} < \infty\}$.
The first step toward proving Lemma \ref{lemma:G} is

\begin{lemma}
\label{lemma:lip} Let $\vf: M \to {\mathbb R}$ be Lipschitz. Then
$\tphi := \vf \circ \pi \in \mbox{Lip}^u(\Delta)$ with
$|\tphi|^u_{\lip} \leq C {\rm Lip}(\vf)$.
\end{lemma}

\begin{proof} Recall that for $x, y \in M$ lying in a piece of
local unstable manifold, we have $d(x,y) \leq p(x,y)^{1/2}$ where
$p(\cdot, \cdot)$ is the $p$-metric (see Sect.~2.3). Now for
$\gamma \in \Gamma^u(\Delta)$ and $x,y \in \gamma$, we have
$$
|\tphi(x)-\tphi(y)| =  |\vf(\pi x) - \vf(\pi y)| \leq {\rm Lip}(\vf) d(\pi x, \pi y)
\le {\rm Lip}(\vf) p(\pi x, \pi y)^{\frac12}\ .
$$
By {\bf (P4)(a)}, $p(\pi x, \pi y)^{\frac12} \le C \alpha^{s(\pi x,\pi y)/2}$, which is
$ \le C d_\beta(x,y)$ since $s \leq s_0$ and $\beta \geq \sqrt{\alpha}$.
\end{proof}

\begin{proof}[Proof of Lemma \ref{lemma:G}]

\medskip
(i)  First we show $\bpi_* \teta = \brho \, \bm$ with $\brho \in \B$.
Let $\psi = \frac{d\eta}{d\nu}$. Then disintegrating $\teta$ into
$\teta^s$ and $\{\rho_\gamma dm_\gamma, \gamma \in \Gamma^u(\Delta)\}$,
we obtain
$$
\rho_\gamma := \tpsi \cdot \frac{d\tnu}{d\mu_\gamma} \cdot
\frac{d\mu_\gamma}{dm_\gamma} \ .
$$
Now $\tpsi$ is bounded by assumption and is $\in  \mbox{Lip}^u(\Delta)$
by Lemma~\ref{lemma:lip}, $\frac{d\tnu}{d\mu_\gamma}$ is
bounded and is $\in  \mbox{Lip}^u(\Delta)$ ([Y], Sect.~2), as is
$\frac{d\mu_\gamma}{dm_\gamma}$
(Lemma~\ref{lemma:jacobian}). Thus we conclude
that $\rho_\gamma \in \mbox{Lip}^u(\Delta)$ and is bounded.
Recall that $\brho(x) = \int_{\gamma^s(x)} \rho_\gamma \, d\teta^s$.
It follows immediately that $|\brho|_\infty \le \sup_\gamma |\rho_\gamma|_\infty$
and Lip$(\brho) \le \sup_\gamma {\rm Lip}^u(\rho_\gamma)$.

\medskip \noindent
(ii) It remains to show $d(\brho) > 0$.
By definition of $\G$, $\psi >0$ on $M^\infty$, the set of points
which never escape from $\hM$, so
$\tpsi >0$ on $\Delta^\infty$.  The fact that
$d\tnu/d\mu_\gamma$ and $e^{-u}$ are strictly positive implies that $\brho > 0$ on
$\Delta^\infty$; hence it is $>0$ on a surviving cylinder set, i.e. a set
$E_k$ such that $\barF^k$ maps $E_k$ onto a surviving $\bDelta_{\ell,j}$
before any part
of it enters the hole. By Corollary~\ref{cor:positive}, $d(\bLp^k\brho) >0$.
Since $\int_{\bDelta^n} g \, d\bm = \int_{\bDelta} \bLp^n g \, d\bm$ for
each $n \geq 0$ and $g \in L^1(\bm)$, we have $d(\bLp^k\brho) = \ra_*^kd(\brho)$ so
that $d(\brho)>0$ as well.
\end{proof}

\begin{proof}[Proof of Theorems 1 and 2.]
Given $\eta \in \G$, let $\teta$ be as defined earlier. Then $\teta \in \tG$
by Lemma \ref{lemma:G}.
For $\vf \in C^0(M)$, let $\tphi = \vf \circ \pi$.
Then $\tphi \in C^0_b(\Delta)$ and
for $n \geq 0$ we have,
\begin{equation}
\label{eq:push}
\hf_*^n \eta(\vf) = \eta(\vf \circ f^n \cdot 1_{M^n})
= \teta(\tilde \vf \circ F^n \cdot 1_{\Delta^n})
= \F_*^n\teta(\tphi) .
\end{equation}
Setting $\vf \equiv 1$ in \eqref{eq:push}, we have
$\eta(M^n) = \hf^n_*\eta(\hM) = \F^n_*\teta (\hDelta) = \teta(\Delta^n)$ for
$n>0$, so
\[
\lim_{n\to \infty} \frac{1}{n} \log \eta(M^n)
= \lim_{n\to \infty} \frac{1}{n} \log \teta(\Delta^n) = \log \ra_*
\]
by Theorem~\ref{thm:tower}(a). This proves Theorem~\ref{thm:escape rate}.

Let $\mu_* = \pi_* \tmu_*$ where $\tmu_*$ is given by Theorem~\ref{thm:tower}.
Then $\mu_*(\vf) = \tmu_*(\tilde \vf)$, and
$$
\hf_*\mu_*(\vf) = \F_*\tmu_*(\tphi) = \ra_* \tmu_*(\tphi) = \ra_* \mu_*(\vf),
$$
proving $\mu_*$ is conditionally invariant.
Using \eqref{eq:push} again, the fact that the normalizations are equal, and
Theorem~\ref{thm:tower}(b), we obtain
\[
\lim_{n\to \infty} \frac{\hf_*^n \eta(\vf)}{\hf_*^n \eta(\hM)}
= \lim_{n \to \infty} \frac{\F_*^n\teta(\tphi)}{\F_*^n\eta(\hDelta)}
= \tmu_*(\tphi) = \mu_*(\vf)\ .
\]
Thus $\hf_*^n \eta/\eta(M^n) \to \mu_*$ weakly.
Finally,
\[
\lim_{n\to \infty} \ra_*^{-n} \hf_*^n \eta(\vf) = \lim_{n\to \infty} \ra_*^{-n}
\F_*^n \teta(\tphi)
= d(\brho) \cdot \tmu_*(\tphi) = d(\brho) \cdot \mu_*(\vf) ,
\]
where $d(\brho)>0$ since $\teta \in \tG$. This completes the proof of Theorem 2.
\end{proof}

\begin{remark}
\label{rem:generalize} {\rm
In the proof of Lemma~\ref{lemma:G}, step (i) holds
for any $\eta$ that has Lipschitz densities on unstable leaves.  Thus for
this class of measures, Theorem~\ref{thm:accim}(b) holds (with $c(\eta)$ possibly
equal to zero).  It is also clear from step (ii) that to show $d(\brho)>0$,
it suffices to assume $\psi >0$ on $M^\infty \cap \Lambda$, or on
$M^\infty \cap \pi(\dlj)$ where $\bpi(\dlj)$ is any surviving element.}
\end{remark}

%%%%%%%%%%%%%%%%%%%%%%%%%%%%%%%%%%%%%%%%%%%%

\subsection{Proof of Theorem~\ref{thm:geometry}}

Let $\mu_* = \pi_* \tmu_*$ be as above.

\medskip
\noindent (a) Since $\hf_* \mu_* = \ra_* \mu_*$, it follows that  $\mu_*$ is supported on
$M \setminus \cup_{n \ge 0} f^n(H)$ where $H=H_\sigma$.
This set has Lebesgue measure zero since
by the ergodicity of $f$, $\cup_{n \ge 0} f^n(H)$ has full Lebesgue measure.
Thus $\mu_*$ is singular with respect to Lebesgue measure.

\medskip
\noindent
(b) First, we argue that $\mu_*$ has absolutely continuous conditional measures
on unstable leaves (without claiming that the densities are strictly positive).
This is true because for each $\ell,j$, $\tilde \mu_*|_{\dlj}$ has absolutely continuous conditional measures on $\gamma \in \Gamma^u(\dlj)$, and
$\pi|_{\dlj}$, which is one-to-one, identifies each $\gamma$
with a positive Lebesgue measure subset of a local unstable manifold of $f$.

The rest of the proof is concerned with showing that the conditional densities
of $\mu_*$ are strictly positive. To do that, it is not productive to view $\mu_*$
as $\pi_*\tmu_*$. Instead, we will view $\mu_*$ as the weak limit of
$\nu^{(n)}:= \hf^n_*\nu/|\hf^n_*\nu|$ as $n \to \infty$ where
$\nu$ is the natural invariant measure for $f$. This convergence of $\nu^{(n)}$ is
guaranteed by Theorem 2.
We will prove that $\mu_*$ has the properties immediately following the
statement of Theorem 3 in Sect.~1.3.

Step 1: Our first patch is built on $V=\cup \{\gamma^u: \gamma^u \in
\Gamma^u(\Lambda)\}$ where $\Gamma^u=\Gamma^u(\Lambda)$ is the
defining family of unstable curves for $\Lambda$. To understand the
geometric properties  of $\nu^{(n)}|_V$, observe that in backward
time, each $\gamma^u \in \Gamma^u$ either falls into the hole
completely or stays out completely. This is because $f(\partial H)$
is regarded as part of the discontinuity set for $f^{-1}$ when we
constructed the horseshoe $\Lambda$ (see Sect.~3.2). Thus there is a
decreasing sequence of sets $U_n = \cup\{ \gamma^u \in \Gamma^u :
f^{-i}\gamma^u \cap H = \emptyset \mbox{ for all } 0 \leq i \leq n
\} \subset V$ consisting of whole $\gamma^u$-curves. Assuming
$\nu(U_n) >0$ for now, we have $\nu^{(n)}|_V = c_n \nu|_{U_n}$ for some
constant $c_n>0$ as $\nu$ is $f$-invariant. Let $\zeta$ be a limit
point of $\nu^{(n)}|_V$, {\it i.e.}, $\zeta = \lim_{n_k}
\nu^{(n_k)}|_V$. Assuming $\zeta(V)>0$, lower bounds for conditional probability
densities of $\nu^{(n_k)}|_V$, equivalently those of $\nu|_{U_n}$, are passed
to $\zeta$, and these bounds are strictly positive.

To see that $\zeta(V)>0$, recall that $\nu = \pi_*\tnu$ for some
$\tnu$ on the tower $\Delta$, so that
$\nu^{(n)} = \pi_* \tnu^{(n)}$
where $\tnu^{(n)} = \F^n_*\tnu/|\F^n_*\tnu|$. Since $\pi(\Delta_0)
\subset V$, we have
$$
\zeta = \lim_{n_k} \nu^{(n_k)}|_V \geq
\lim_{n_k} \pi_*( \tnu^{(n_k)} |_{\Delta_0})
= \pi_*(\tilde \mu_*|_{\Delta_0})\ .
$$
We have written an inequality (as opposed to equality) above because parts of
$\Delta_\ell$ for $\ell \ge 1$ may get mapped into $V$ as well.
Clearly, $\tmu_*(\Delta_0) >0$, thereby ensuring $\zeta(V) >0$, hence
$\nu(U_n) >0$ and
the strictly positive conditional densities property above.
This together with $\zeta \le \mu_*|_V$ (equality
is not claimed because it is possible for part of $\nu^{(n_k)}$ from
outside of $V$ to leak into $V$ in the limit) proves that
$(V,\zeta)$ is an acceptable patch.

Step 2: Next we use $(V,\zeta)$ to build patches $(V_{\ell,j}, \zeta_{\ell,j})$
corresponding to partition elements $\dlj$ of the tower $\Delta$ with $\ell>0$
and $\tmu_*(\dlj) >0$. From Sections~3 and 4, we know that $\pi(\dlj)$ is a hyperbolic product set, and $\pi(\dlj) = f^\ell(\Lambda^s)$ for some $s$-subset
$\Lambda^s \subset \Lambda$. Moreover, $f^i(\Lambda^s) \cap H =
\emptyset$ for all $0 < i \le \ell$. Thus we may assume
$V_{\ell,j} = \cup \{\gamma^u : \gamma^u \in \Gamma^u(\pi(\dlj))\}
\subset f^\ell(V)$. Let $\zeta_{\ell,j} = \ra_*^{-\ell}(f^\ell_*\zeta)|_{V_{\ell,j}}$.
Then $\zeta_{\ell,j}$ has strictly positive conditional densities on unstable curves
because $\zeta$ does,
and $\zeta_{\ell,j} \le \mu_*|_{V_{\ell,j}}$ as
$\mu_*$ satisfies $\hf_*\mu_* = \ra_* \mu_*$.

Finally, since $\zeta_{\ell,j} \ge \pi_*(\tilde \mu_*|_{\dlj})$ for each $\ell,j$,
it follows that $\sum_{\ell,j} \zeta_{\ell,j} \ge \mu_*$, completing the
proof of Theorem 3.

%%%%%%%%%%%%%%%%%%%%%%%%%%%%%%%%%%%%%%%%%

\subsection{Proof of  Theorem~\ref{thm:small hole limit}}

Suppose $h_n$ is a sequence of numbers tending
to $0$, $\sigma_{h_n} \in \Sigma_{h_n}(q_0)$ is a sequence
of holes in the billiard table, and $H_n = H_{\sigma_{h_n}}$ the corresponding
holes in $M$. For each $n$, let $\ra_n$ be the escape rate and
$\mu_n$ the physical measure for the leaky system $(f,M,H_n)$
given by Theorem~\ref{thm:accim}. By Proposition~\ref{prop:uniform bounds}(iii),
we have
$\ra_n \to 1$ as $n\to \infty$. To prove $\mu_n \to \nu$, we will
assume, having passed to a subsequence, that $\mu_n$ converges weakly
to some $\mu_\infty$, and show that
(i) $\mu_\infty$ is $f$-invariant, and (ii) it has absolutely continuous
conditional measures on unstable leaves. These two properties together
uniquely characterize $\nu$.

The following notation will be used:
$\Lambda(n)$ is the generalized horseshoe respecting the hole
$H_n$, $\Delta(n) = \cup_\ell \Delta_\ell(n)$ is the corresponding tower,
$F_n: \Delta(n) \to \Delta_n$ is the tower map,
$\pi_n: \Delta(n) \to M$ is
the projection, and $\tmu_n$ is the conditionally
invariant measure on $\Delta(n)$ that projects to $\mu_n$.

\medskip
\noindent (i) {\it Proof of $f$-invariance:} Let $S = \partial M \cup f^{-1}\partial M$.

\begin{lemma}
\label{lemma:singularity}
$\mu_\infty(S)=0$
\end{lemma}

\begin{proof} Let $\delta_1$ and $\lambda_1$ be as in Sect.~2.3, and let
$N_\ve(S)$ denote the $\ve$-neighborhood of $S$. We claim that there
exist constants $C_3, \varsigma > 0$ such that for $\ve<\delta_1$ and for all $n$,
$\mu_n(N_\ve(S)) \leq C_3 \ve^\varsigma$.
By the construction of $\Lambda=\Lambda(n)$, any $n$,
 $d(f^\ell(\Lambda), S) \geq \delta_1 \lambda_1^{-\ell}$.
 Thus $f^\ell(\Lambda) \cap N_\ve(S) = \emptyset$ for all
 $\ell  \le - \log (\ve/\delta_1) / \log \lambda_1$. Hence
\[
\mu_n(N_\ve(S)) \leq \; \sum_{\ell >  - \log (\ve/\delta_1) / \log \lambda_1}
\tmu_n(\Delta_\ell(n))\ ,
\]
which by Proposition~\ref{prop:uniform bounds}(ii) is
$\leq K (\beta^{-1}\theta_0)^{- \log (\ve/\delta_1) / \log \lambda_1}$, proving
the claim above with $C_3 = K/\delta_1$
and $\varsigma = \log (\beta \theta_0^{-1})/\log \lambda_1$.
Since $C_3$ and $\varsigma$ are independent of $n$, these bounds
pass to $\mu_\infty$, implying $\mu_\infty(S)=0$.
\end{proof}

Having established that $f$ is well defined $\mu_\infty$-a.e., we
now verify that $\mu_\infty$ is $f$-invariant: Let
$\varphi:M \to {\mathbb R}$ be a continuous function. Then
$$
\int (\varphi \circ f) d\mu_\infty = \lim_{n\to \infty} \int (\varphi \circ f) d\mu_n
=  \lim_{n\to \infty} \int \varphi \ d(f_*\mu_n)\ ,
$$
and
\begin{equation}
\int \varphi \ d(f_*\mu_n) = \int_{M \setminus H_n} \varphi \ d(f_*\mu_n)
+ \int_{H_n} \varphi \ d(f_*\mu_n)\ .
\label{eq:int}
\end{equation}
Since $(f_*\mu_n)|_{M \setminus H_n} = \mathring f_*\mu_n = \ra_n \mu_n$, the first integral on the
right side of (\ref{eq:int}) is equal to $\ra_n \int \varphi \, d\mu_n$, while the
absolute value of the second is bounded by
$ (1-\ra_n) |\varphi|_\infty$.
Since $\ra_n \to 1$ as $n \to \infty$,
the right side of
(\ref{eq:int}) tends to $\int \varphi \, d\mu_\infty$.

\medskip
\noindent {\it (ii) Absolutely continuous conditional measures on unstable leaves:}
Since the measures $\tmu_n$ do not live on the same space for different $n$,
a first task here is to find common domains in $M$ on which $(\pi_n)_*\tmu_n$
can be compared. In the constructions to follow, the discontinuity set refers to
the real discontinuity set of $f$, {\it not} the ones that include
 boundaries of holes (as was done in Sect.~3).

We choose a rectangular region
$\check Q$ slightly larger than $Q$ in Proposition 2.2, large enough that
$\check Q \supset \Lambda(n)$ for all $n$, and let $\check \Gamma^u$
denote the set of all homogeneous unstable curves connecting the
two components of $\partial^s \check Q$.
Let $\check V = \cup \{\gamma^u
\in \check \Gamma^u\}$. Then $\Lambda(n) \subset \check V$ for all
$n$, for $\gamma \cap \check Q \in \check \Gamma^u$ for
every $\gamma \in \Gamma^u(\Lambda(n))$ (defined using the enlarged
discontinuity set).
Now for all $n$, $(\pi_n)_*(\tilde \mu_n|_{\Delta_0(n)})$ is a sequence
of measures on $\check V$ with
absolutely continuous conditional measures on the elements of
$\check \Gamma^u$.  Moreover, the conditional densities are uniformly bounded
from above with a bound independent of $n$ (Proposition~\ref{prop:uniform bounds}(i)).
Let $\mu_{\infty,0}$ be a limit point of $(\pi_n)_*(\tilde \mu_n|_{\Delta_0(n)})$.
Assuming $\mu_{\infty,0}(\check V) >0$, these density bounds are inherited
by $\mu_{\infty,0}$. To show  $\mu_{\infty,0}(\check V) >0$, we will argue
there exists $b>0$ such that
$\tilde \mu_n(\Delta_0(n))>b$ for all $n$, and that is true because the
$\tilde \mu_n$ are probability measures,
there is a uniform lower bound on $\tilde \mu_n(\cup_{\ell <L} \Delta_\ell(n))$
for large enough $L$ (Proposition~\ref{prop:uniform bounds}(ii)), and
$\tilde \mu_n(\Delta_{\ell+1}(n)) \le \ra_n^{-1}\tilde \mu_n(\Delta_\ell(n))$.

For $\ell>0$, we define $\check Q_\ell$ to be the finite union of
$s$-subrectangles of $\check Q$ retained in $\ell$ steps in the construction
of $\Lambda$ {\it when $f$ has no holes}, i.e. roughly speaking,
$\check Q_\ell$ consists of points that stay away from
$S = \partial M \cup f^{-1} \partial M$ by a distance $ \ge \delta_1 \lambda_1^{-i}$
at step $i$. Let $\check V_\ell = \check V \cap \check Q_\ell$.
Then $\pi_n(\F_n^{-\ell} \Delta_\ell(n)) \subset \check V_\ell$ for all $n$.
In fact, for each $j$, $\pi_n(\F_n^{-\ell}\dlj(n))$ is contained in a connected
component of $\check Q_\ell$. The argument for $\mu_{\infty,0}$ can now
be repeated to conclude the existence of a limit point of
$(\pi_n \circ \F_n^{-\ell})_* (\tilde \mu_n|_{\Delta_\ell(n)})$ with
absolutely continuous conditional
measures on unstable leaves. Pushing all measures
forward by $f_*^\ell$, this gives a limit point $\mu_{\infty, \ell}$ of
$(\pi_n)_*(\tilde \mu_n|_{\Delta_\ell(n)})$  as $n \to \infty$ with the same property.

To proceed systematically, we perform a Cantor diagonal argument,
choosing a single subsequence $n_k$
with the property that for each $\ell \ge 0$,
$(\pi_{n_k})_*(\tilde \mu_{n_k}|_{\Delta_\ell(n_k)})$ converges to a measure
$\mu_{\infty,\ell}$ on $f^\ell \check V_\ell$. Finally, to conclude $\mu_\infty =
\sum_\ell \mu_{\infty,\ell}$, we need a tightness condition as the towers are
noncompact. This is given by Proposition~\ref{prop:uniform bounds}(ii).

\medskip
The proof of Theorem 4 is now complete.

%%%%%%%%%%%%%%%%%%%%%%%%%%%%%%%%%%%%%%%%%%%%%%%%%%%%%%%%%%%%%%%%%%%%%%%%%%%%%%%%%

%%%%%%%%%%%%%%%%%%%%%%%%%%%%%%%%%%%%%%%%%%%%%%%%%%%%%%%%%%%%%%%%%%%%%%%%%%%%%%%%%

\small
%\bibliography{billiard_hole}

\end{document}